\documentclass[11pt]{amsart}
\usepackage[all]{xy}
\usepackage{amscd}
\usepackage{amsfonts}
\usepackage{amsmath}
\usepackage{amssymb}
\usepackage{amsthm}
\usepackage{bm}
\usepackage{CJK}
\usepackage{dsfont}
\usepackage{fancyhdr}
\usepackage{graphics}
\usepackage{graphicx}
\usepackage{latexsym}
\usepackage{mathrsfs}
\usepackage{parskip}
\usepackage{slashed}
\usepackage{titletoc}
\usepackage[usenames]{color}
\allowdisplaybreaks

\begin{CJK*}{GBK}{kai}
\newtheorem{thm}{Theorem}[section]
\end{CJK*}

\newtheorem{cor}[thm]{Corollary}
\newtheorem{conj}[thm]{Conjecture}
\newtheorem{lem}[thm]{Lemma}
\newtheorem{prop}[thm]{Proposition}

\newtheorem{defn}[thm]{Definition}
\numberwithin{equation}{section}
\newtheorem{rem}[thm]{Remark}

\newcommand{\ric}{\mathrm{Ric}}
\newcommand{\tr}{\mathrm{tr}}
\newcommand{\ppt}{\frac{\partial}{\partial t}}
\newcommand{\ddt}{\frac{\mathrm{d}}{\mathrm{d} t}}
\newcommand{\mn}{\sqrt{-1}}
\newcommand{\de}{\partial}
\newcommand{\dbar}{\overline{\partial}}
\newcommand{\ddbar}{\sqrt{-1}\partial\overline{\partial}}
\newcommand{\md}{\mathrm{d}}

\begin{document}
\begin{CJK}{GBK}{song}
\title{A parabolic Monge-Amp\`ere type equation of Gauduchon metrics}
\makeatletter
\let\uppercasenonmath\@gobble
\let\MakeUppercase\relax
\let\scshape\relax
\makeatother
\author{Tao Zheng }
\address{Institut Fourier, Universit\'{e} Grenoble Alpes, 100 rue des maths,
Gi\`{e}res 38610, France}
\email{zhengtao08@amss.ac.cn}
\subjclass[2010]{53C55, 35J60, 32W20, 58J05}
\thanks {Supported by National Natural Science Foundation of China grant No. 11401023, and the author's post-doc is supported by the European Research Council (ERC) grant No. 670846 (ALKAGE)}
\keywords{Gauduchon metric,  (parabolic) Monge-Amp\`ere type equation, $C^{2,\gamma}$ estimate, Liouville type theorem}
\begin{abstract}
We prove the long time existence and uniqueness of solution to a parabolic Monge-Amp\`ere type equation on compact Hermitian manifolds. We also show that the normalization of the solution converges to a smooth function in the smooth topology as $t$ approaches infinity which, up to scaling, is the solution to a Monge-Amp\`ere type equation. This gives a parabolic proof of the Gauduchon conjecture based on the solution of Sz\'ekelyhidi, Tosatti and Weinkove to this conjecture.
\end{abstract}
\maketitle
\section{Introduction}
Let $(M,\alpha)$ be a compact complex Hermitian manifold with $\dim_{\mathbb{C}}M=n\geq 2$. Then the real $(1,1)$ form associated to the Hermitian metric $\alpha$ (denoted by itself) is defined by
$$
\alpha=\mn\alpha_{i\overline{j}}\md z^i\wedge\md \overline{z}^j.
$$
The Hermitian metric $\alpha$ is called
\emph{K\"{a}hler} if
$
\md \alpha=0,
$
\emph{Astheno-K\"{a}hler} (see \cite{jostyau}) if
$
\partial\dbar\alpha^{n-2}=0,
$
\emph{balanced} (see \cite{michelsohn}) if
$
\md\alpha^{n-1}=0,
$
\emph{Gauduchon} (see \cite{gauduchon1}) if
$
\partial\dbar\alpha^{n-1}=0,
$
and \emph{strongly Gauduchon} (see \cite{popovici}) if
$
\dbar\alpha^{n-1}\;\text{is}\;\partial\textrm-\text{exact}.
$

If $(M,\alpha)$ is a K\"{a}hler manifold, then Yau's solution \cite{yau1978} to the Calabi conjecture (see \cite{Cao} for a parabolic proof using the estimates in \cite{yau1978}) says that for any given smooth positive volume form $\sigma$ on $M$ satisfying $\int_M\sigma=\int_M\alpha^n$, there exists a unique K\"{a}hler metric $\omega$ with $[\omega]=[\alpha]\in H^{2}(M,\mathbb{R})$ such that
\begin{align}\label{volumform}
\omega^n=\sigma.
\end{align}
 Moreover, Yau's theorem is equivalent to the following statement. Given any $\Psi\in c_{1}(M)$, the first Chern class, we can find a unique K\"{a}hler metric $\omega$ with $[\omega]=[\alpha]\in H^{2}(M,\mathbb{R})$ such that
 \begin{align}\label{ricciform}
 \ric(\omega)=\Psi,
 \end{align}
where $\ric(\omega)$ is the Ricci form of the K\"{a}hler metric $\omega$ and can be defined  as
$$
\ric(\omega)=-\ddbar\log\det\omega.
$$

It is natural to ask whether there hold similar results when $M$ does not admit a K\"{a}hler metric, but only a Hermitian metric $\alpha$.
If there is no restriction on the class of Hermitian metrics, then we can solve \eqref{volumform} trivially by a conformal change of metric.

Tosatti and Weinkove \cite{TW4} proved that for any Hermitian metric $\alpha$ on $M$, there exists a Hermitian metric $\omega$ of the form $\omega=\alpha+\ddbar u$ with $u\in C^{\infty}(M,\mathbb{R})$ such that \eqref{ricciform} holds. (cf. \cite{Ch,gill,TW1} and see \cite{gill} for a parabolic proof based on \cite{TW4}). Chu, Tosatti and Weinkove \cite{ctw} recently proved similar results on almost Hermitian manifolds, based on which Chu \cite{chu1607} gave a parabolic proof.

Gauduchon \cite{gauduchon1} showed that there exists a unique Gauduchon metric up to scaling (when $n\geq 2$) in  the conformal class of any Hermitian metric $\alpha$.

Motivated by Yau's solution to the Calabi conjecture, Gauduchon \cite[Chapter IV.5]{gauduchon2} proposed the following conjecture.
\begin{conj}[Gauduchon \cite{gauduchon2}; 1984]\label{conj1}
Let $M$ be a compact Hermitian manifold and $\Psi$ be a closed real $(1,1)$ form on $M$ with $[\Psi]=c_{1}^{\mathrm{BC}}(M)\in H^{1,1}_{\mathrm{BC}}(M,\mathbb{R})$. Then there exists a Gauduchon metric satisfying \eqref{ricciform},
where $\ric(\omega)$ is the Chern-Ricci form of $\omega$ and can be expressed as
$$\ric(\omega)=-\ddbar\log\det\omega.$$
\end{conj}
Here $H^{1,1}_{\mathrm{BC}}(M,\mathbb{R})$ is the \emph{Bott-Chern cohomology group} and its definition  can be found in Section \ref{preliminary}.

This conjecture can also be restated as follows.
\begin{conj}\label{conj2}
Let $M$ be a compact complex Hermitian manifold and $\sigma$ be a smooth positive volume form. Then there exists a Gauduchon metric $\omega$ on $M$ satisfying \eqref{volumform}.
\end{conj}
Very recently, Sz\'ekelyhidi, Tosatti and Weinkove \cite{stw1503} solved this conjecture based on their previous works \cite{sz,tw1305,tw1310}. More precisely, they proved
\begin{thm}[Sz\'ekelyhidi, Tosatti and Weinkove \cite{stw1503}; 2015]\label{stwthm}
Let $M$ be a compact complex Hermitian manifold with a Gauduchon metric $\alpha_0$, and $\Psi$ be a closed real $(1,1)$ form on $M$ with $[\Psi]=c_{1}^{\mathrm{BC}}(M,\mathbb{R})\in H^{1,1}_{\mathrm{BC}}(M,\mathbb{R})$. Then there exists a Gauduchon metric $\omega$ with  $[\omega]_A=[\alpha_{0}^{n-1}]_A\in H_{A}^{n-1,n-1}(M,\mathbb{R})$ solving \eqref{ricciform}.
\end{thm}
Here $H_{A}^{n-1,n-1}(M,\mathbb{R})$ is the \emph{Aeppli cohomology group} and its definition  can be found in Section \ref{preliminary}.

Tosatti and Weinkove \cite{tw1310} deduced that to obtain this theorem it is sufficient to solve the following  partial differential equation, which was also independently introduced by Popovici \cite{Po}.
They sought a Hermitian metric $\omega$ on $M$ with the property that
\begin{align}\label{omegan-1}
\omega^{n-1}=&\alpha_0^{n-1}+\partial\left(\frac{\mn}{2}\dbar u \wedge\alpha^{n-2}\right)+\overline{\partial\left(\frac{\mn}{2}\dbar u \wedge\alpha^{n-2}\right)}\\
=&\alpha_0^{n-1}+\ddbar u \wedge\alpha^{n-2}+\mathrm{Re}\left(\mn \partial u \wedge\dbar(\alpha^{n-2})\right),\nonumber
\end{align}
where $u \in C^\infty(M,\mathbb{R})$ and $\alpha$ is a background Gauduchon metric. If $\alpha_0$ is Gauduchon, then the metric $\omega$ is also Gauduchon. Note that there exists an $F\in C^{\infty}(M,\mathbb{R})$ such that $\ric(\alpha)=\Psi+\ddbar F$.
Now we can deduce that \eqref{omegan-1} is equivalent to
$$
\omega^n=e^{F+b}\alpha^n
$$
with some constant $b\in\mathbb{R}$. Note that Sz\'ekelyhidi, Tosatti and Weinkove \cite{stw1503} solved the following equivalent equation (their paper solved a family of Monge-Amp\`ere type equations including this one)
$$
\log\frac{\det\left(\ast\frac{\omega^{n-1}}{(n-1)!}\right)}{\det\alpha}
=\log\left(\frac{\det\omega}{\det\alpha}\right)^{n-1}=(n-1)(F+b),
$$
i.e.,
\begin{align}
\label{stw1503equ}
\log\frac{\left(\varpi+\frac{1}{n-1}\left[(\Delta u)\alpha-\ddbar u\right]+Z(u)\right)^n}{\alpha^n}=(n-1)(F+b),
\end{align}
where
$
\varpi=\frac{1}{(n-1)!}\ast\alpha_0^{n-1},
$
$
\Delta u=\alpha^{\overline{j}i}\partial_i\partial_{\overline{j}}u,
$
\begin{align}
\label{tildeomega}
\tilde{\omega}:=\varpi+\frac{1}{n-1}\left[(\Delta u)\alpha-\ddbar u\right]+Z(u)
=:\mn \tilde g_{i\overline{j}}\md z^i\wedge\md \overline{z}^j>0,
\end{align}
and
\begin{align}
\label{zu}
Z(u)=\frac{1}{(n-1)!}\ast\mathrm{Re}\left[\mn\de u\wedge\dbar(\alpha^{n-2})\right].
\end{align}
This question is a variant of the one introduced by Fu, Wang and Wu \cite{FWW1} and also related to the notion of $(n-1)$-plurisubharmonic (Psh) functions (see \cite{HL1,HL2}).

More remarks about Conjecture \ref{conj1} (also Conjecture \ref{conj2}) and applications of the methods in the proof of Theorem \ref{stwthm} can be found in \cite{stw1503} and references therein.

In this paper, for any $\psi\in C^{\infty}(M,\mathbb{R})$, we consider a parabolic version of \eqref{stw1503equ}, analogs to \cite{Cao,gill,chu1607}, as follows.
\begin{align}
\label{paragau}
\ppt u=\log\frac{\left(\varpi+\frac{1}{n-1}\left[(\Delta u)\alpha-\ddbar u\right]+Z(u)\right)^n}{\alpha^n}-\psi
\end{align}
with $u(0)=u_0\in C^{\infty}(M,\mathbb{R})$. In the following, we use \eqref{tildeomega} and \eqref{zu} with $u$ evolved by \eqref{paragau}.

It is easy to deduce that \eqref{paragau} is equivalent to the following flow
\begin{align*}
\ppt \alpha_t^{n-1}
=&-(n-1)\left(\ric(\alpha_t)-\ric(\alpha)+\frac{1}{n-1}\ddbar \psi\right)\wedge\alpha^{n-2}\\
&+(n-1)\mathrm{Re}\left(\mn\de\left(\log\frac{\alpha_t^n}{e^{\psi/(n-1)}\alpha^n}\right)\wedge\dbar(\alpha^{n-2})\right),
\end{align*}
with initial metric
$$
\alpha(0)^{n-1}=\alpha_0^{n-1}+\ddbar u_0 \wedge\alpha^{n-2}+\mathrm{Re}\left(\mn \partial u_0 \wedge\dbar(\alpha^{n-2})\right)>0.
$$
This flow  preserves the Gauduchon condition if the initial metric $\alpha_0$ is Gauduchon.
Indeed, taking $\partial \dbar$ on both sides of \eqref{pshflow2}, we can obtain
$$
\ppt \de\dbar\alpha^{n-1}_t=(n-1)\de\dbar\left(\partial\gamma+\overline{\partial\gamma}\right)=0,
$$
where
$$
\gamma=\frac{\mn}{2}\dbar \left(\log\frac{\alpha_t^n}{e^{\psi/(n-1)}\alpha^n}\right)\wedge\alpha^{n-2},
$$
as required. When $n=2$ this flow can be seen as the ``twisted"  Chern-Ricci flow (cf. \cite{ftwz,gill,gillmmp,gillscalar,twjdg,twcomplexsurface,twymathann,yangxiaokui,zhengtaocjm}). We show
\begin{thm}\label{mainthm}
Let $(M,\alpha_0)$ be a compact Hermitian manifold with $\dim_{\mathbb{C}}M=n\geq 3$ and $\alpha$ be a Gauduchon metric on $M$. Then there exists a unique solution $u$ to \eqref{paragau} on  $M\times[0,\,\infty)$ and if we define the normalization of $u$ by
$$\tilde u(x,t):=u(x,t)-\frac{1}{\mathrm{Vol}_{\alpha}(M)}\int_Mu(y,t)\alpha^{n}(y),$$
then $\tilde u$ converge smoothly to a function $\tilde u_{\infty}$ as $t\longrightarrow \infty$, and $\tilde u_{\infty}$ is the unique solution to \eqref{stw1503equ} by taking $\psi=(n-1)F$, up to adding a constant $\tilde b\in\mathbb{R}$ defined as in \eqref{btilde}.
\end{thm}
This gives a parabolic
proof of the Gauduchon conjecture based on the solution of Szekelyhidi,
Tosatti and Weinkove to this conjecture \cite{stw1503}. It is
analogous to H.-D. Cao's parabolic proof of the Calabi conjecture \cite{Cao}
based on Yau's work \cite{yau1978}, and to Gill's result in the Hermitian case \cite{gill} based
on Tosatti and Weinkove \cite{TW4}.

\begin{rem}
We can also consider another kind of
parabolic flow of Gauduchon metrics, a revised version of Gill \cite{gill1410}, given by
\begin{align}\label{pshflow2}
\ppt \alpha^{n-1}_t=-(n-1)\ric(\alpha_t)\wedge \alpha^{n-2}+(n-1)\mathrm{Re}\left(\mn\de\left(\log\frac{\alpha_t^n}{\alpha^n}\right)\wedge\dbar(\alpha^{n-2})\right)
\end{align}
with $\alpha(0)=\alpha_0$. Note that in the case $n=2$ this flow is exactly the Chern-Ricci flow.

If the initial metric $\alpha_0$ is Gauduchon and $\alpha$ is Astheno-K\"{a}hler, then it is easy to deduce that the flow \eqref{pshflow2} preserves the Gauduchon condition. Taking Aeppli cohomology, we can deduce
\begin{equation*}
\ddt [\alpha_t^{n-1}]_A=-(n-1)[\ric(\alpha)\wedge\alpha^{n-2}]_A,
\end{equation*}
the right side of which is independent of time $t$. Note that Gill \cite{gill1410} also introduced a parabolic flow suggested by Tosatti and Weinkove \cite{tw1305,tw1310}
\begin{align}\label{pshflow1}
\ppt \alpha^{n-1}_t=-(n-1)\ric(\alpha_t)\wedge \alpha^{n-2}
\end{align}
which also preserves the Gauduchon condition under the same assumptions as above. However, if we take Aeppli Cohomology on the both sides of \eqref{pshflow1}, then we get
\begin{equation*}
\ddt [\alpha_t^{n-1}]_A=-(n-1)[\ric(\alpha_t)\wedge\alpha^{n-2}]_A,
\end{equation*}
the right side of which is dependent on the time $t$.

If $M$ is non-K\"{a}hler Calabi-Yau manifold, i.e., $c_{1}^{\mathrm{BC}}(M)=0$, then we can take $\alpha$ to be the Chern-Ricci flat Hermitian metric by Tosatti and Weinkove \cite{TW4} and then it is easy to see that the flow \eqref{pshflow2} also preserves the Gauduchon condition.

Using the method in this paper originated from \cite{tw1305,tw1310,stw1503}, it follows that there exists a unique solution to \eqref{pshflow2} on $M\times[0,\,T)$, where
\begin{align*}
T:=&\sup\Big\{t\geq0:\;\exists\; \psi\in C^{\infty}(M,\,\mathbb{R}) \text{ such that}\\
&\quad\quad\quad\quad\quad\quad\Phi_t+\ddbar\psi\wedge\alpha^{n-2}
+\mathrm{Re}\left[\mn\de\psi\wedge\dbar(\alpha^{n-2})\right]>0\Big\},\nonumber
\end{align*}
with
\begin{equation*}
\Phi_t:=\alpha_0^{n-1}-t(n-1)\ric(\alpha)\wedge\alpha^{n-2}.
\end{equation*}
This solves a conjecture revised from \cite[Conjecture 1.2]{gill1410}.
\end{rem}
\begin{rem}
Let $M$ be a compact complex manifold with two Hermitian metrics $\omega_0$ and $\omega$, and $\dim_{\mathbb{C}}M=n$. Then for any $F\in C^{\infty}(M,\mathbb{R})$,  Tosatti and Weinkove \cite{tw1305,tw1310} proved that there exists a unique pair $(u,b)$ with $u\in C^{\infty}(M,\mathbb{R})$ and $b\in \mathbb{R}$ such that
\begin{equation} \label{equtw1310}
\det \left( \omega_{0}^{n-1} + \ddbar u \wedge \omega^{n-2} \right) = e^{F+b} \det \left( \omega^{n-1} \right),
\end{equation}
with $$\omega_0^{n-1} + \ddbar u \wedge \omega^{n-2} > 0, \quad \sup_M u=0.$$
If $\omega$ is K\"{a}hler, then this is a conjecture of Fu and Xiao \cite{FX}(see also \cite{FWW1,FWW2,Po2}).
We can also consider the parabolic version of \eqref{equtw1310}
\begin{equation} \label{tw1305remark1.5}
\ppt{} u = \log \frac{\left( \omega_h + \frac{1}{n-1} ( (\Delta u) \omega - \ddbar u) \right)^n}{\Omega},
\end{equation}
for a fixed volume form $\Omega$ and $\omega_h=\frac{1}{(n-1)!}\ast\omega_{0}^{n-1}$, where $\ast$ is respect to $\omega$.  Tosatti and Weinkove \cite[Remark 1.5]{tw1305} conjectured that the solutions to \eqref{tw1305remark1.5} exist for all time and converge (after normalization) to give solutions to \eqref{equtw1310} up to a constant. Using the method in this paper, we can confirm this conjecture.
\end{rem}
The paper is organized as follows. In section \ref{preliminary}, we collect some basic concepts about Hermitian manifolds.  In section \ref{secpreliminary}, we give the uniform bounds of the normalization of the solution $\tilde u$ to \eqref{paragau}. In section \ref{secsecond} and Section \ref{secfirst}, we give the second and first order priori estimates of the solution $u$ to \eqref{paragau} respectively. In Section \ref{proofofuniqueandlong}, we prove the long time existence and uniqueness of the equation \eqref{paragau} claimed as in the first part of Theorem \ref{mainthm}. In Section \ref{harnack}, we give the Harnack inequality for the equation \eqref{Ldefn} which will be used to prove the convergence of the normalization of the solution to \eqref{paragau} claimed as in the second part of Theorem \ref{mainthm} in Section \ref{convergence}.

\noindent {\bf Acknowledgements}
The author thanks Professor Valentino Tosatti and Professor Ben Weinkove for suggesting him this problem. The warmest thanks goes to Professor Valentino Tosatti for his helping the author overcoming the difficulties in the preparation for this paper, pointing out some mistakes in calculation, offering more references and so many other useful comments on an earlier version of this paper.
The author also thanks Professor Jean-Pierre Demailly for some useful comments and Doctor Wenshuai Jiang for some helpful discussions. The author is also grateful to the anonymous referees and the editor for their careful reading and helpful suggestions which greatly improved the paper.

\section{Preliminaries}\label{preliminary}
In this section, to avoid confusions, we collect some preliminaries about Hermitian geometry which will be used in this paper.

Let $(M,J,g)$ be a Hermitian manifold with $\dim_{\mathbb{C}}M=n$, $J$ be the canonical complex structure and $g$ be the Riemannian metric with $g(JX,\,JY)=g(X,\,Y)$ for any vector fields $X,\,Y\in \mathfrak{X}(M)$. Then in the real local coordinates $(x^1,\cdots,x^{2n})$ with
$$J\left(\partial/\partial x^i\right)=\partial/\partial x^{n+i},\;J\left(\partial/\partial x^{n+i}\right)=-\partial/\partial x^i,\;i=1,\cdots, n,$$
we have
$$
g_{ij}=g_{n+i,n+j},\quad g_{i,n+j}=-g_{n+i,j},\quad g_{\alpha\beta}=g_{\beta\alpha},\quad i,\,j=1,\cdots,n,\;\alpha,\,\beta=1,\cdots,2n,
$$
where $g_{\alpha\beta}=g\left(\partial/\partial x^{\alpha}, \partial/\partial x^{\beta}\right)$. We can define a real $2$-form $\omega$  by
$$
\omega(X,\,Y):=g(JX,\,Y),\quad \forall\;X,\,Y\in \mathfrak{X}(M).
$$
This form is determined uniquely by $g$ and vice versa. We define the volume element as usual by
$$\md V=\sqrt{\mathrm{det}(g_{\alpha\beta})}\md x^1\wedge\cdots\wedge\md x^{2n}.$$
For any $p$ form $\varphi$
$$
\varphi=\frac{1}{p!}\varphi_{i_1\cdots i_p}\md x^{i_1}\wedge\md x^{i_p},
$$
the star $\ast$ operator is defined by
\begin{align}\label{astdefn}
\psi\wedge\ast\varphi=\frac{1}{p!}\psi_{j_1\cdots j_p}g^{j_1 \ell_1}\cdots g^{j_1 \ell_p}\varphi_{\ell_1\cdots \ell_p}\ast 1=\frac{1}{p!}\psi_{j_1\cdots j_p}g^{j_1 \ell_1}\cdots g^{j_1 \ell_p}\varphi_{\ell_1\cdots \ell_p}\md V,
\end{align}
where $\psi=\frac{1}{p!}\psi_{i_1\cdots i_p}\md x^{i_1}\wedge\cdots\wedge\md x^{i_p}$ is another $p$ form. We define inner product by
$$
\langle\psi,\,\varphi\rangle:=\frac{1}{p!}\psi_{j_1\cdots j_p}g^{j_1 \ell_1}\cdots g^{j_1 \ell_p}\varphi_{\ell_1\cdots \ell_p}.
$$
It is easy to deduce that
\begin{align}\label{astdefn2}
\ast \varphi=\frac{1}{p!(2n-p)!}\sqrt{\mathrm{det}(g_{\alpha\beta})}\delta_{j_1\cdots j_p k_1\cdots k_{2n-p}}g^{j_1 \ell_1}\cdots g^{j_1 \ell_p}
\varphi_{\ell_1\cdots \ell_p}\md x^{k_1}\wedge\cdots\md x^{k_{2n-p}},
\end{align}
where $\delta_{j_1\cdots j_p k_1\cdots k_{2n-p}}$ is the general Kronecker symbol. It is easy to get
\begin{align}\label{relationast}
\ast \ast\varphi=(-1)^{2np+p}\varphi,
\end{align}
where note that $\dim_{\mathbb{R}}M=2n$.

In the complex local coordinates $$z=(z^1,\cdots,z^n)=(x^1+\sqrt{-1}x^{n+1},\cdots,x^n+\sqrt{-1}x^{2n}),$$
we denote $\partial_i=\partial/\partial z^i,\;\partial_{\overline{j}}=\partial/\partial\overline{z}^j,\;i,\,j=1,\cdots,n$. Then we have
\begin{align}
g=&\sum\limits_{i,j=1}^n g_{i\overline{j}}\left(\mathrm{d}z^i\otimes \mathrm{d}\overline{z}^j+ \mathrm{d}\overline{z}^j\otimes\mathrm{d}z^i\right),\nonumber\\
\label{kahlerform}\omega=&\mn\sum\limits_{i,j=1}^n g_{i\overline{j}}\mathrm{d}z^i\wedge \mathrm{d}\overline{z}^j,
\end{align}
where $g_{i\overline{j}}=\frac{1}{2} \left(g_{i,j}+\sqrt{-1}g_{i,n+j}\right)$. Then we can choose $\frac{\omega^n}{n!}$ as the volume element. Therefore, for any $(p,q)$ form
$$
\phi=\frac{1}{p!q!}\phi_{i_1\cdots i_p\overline{j_1}\cdots\overline{j_q}}\md z^{i_1}\wedge\cdots\wedge\md z^{i_p}\wedge\md \overline{z}^{j_1}\wedge\cdots\wedge\md \overline{z}^{j_q},
$$
using \eqref{astdefn} and \eqref{astdefn2}, we can deduce (see for example \cite{luqikeng})
\begin{align}\label{starformulacomplex}
\ast \phi =&\frac{(\mn)^n(-1)^{np+\frac{n(n-1)}{2}}\det g}{(n-p)!(n-q)!p!q!}\phi_{i_1\cdots i_p\overline{j_1}\cdots\overline{j_q}}g^{\overline{\ell_1}i_1}
\cdots g^{\overline{\ell_p}i_p}g^{\overline{j_1}k_1}\cdots g^{\overline{j_q}k_q}\\
&\delta_{\ell_1\cdots \ell_pb_1\cdots b_{n-p}}^{1\cdots\cdots n}\delta_{k_1\cdots k_qa_1\cdots a_{n-q}}^{1\cdots\cdots n}
\md z^{a_1}\wedge\md z^{a_{n-q}}\wedge\md \overline{z}^{b_1}\wedge\md\overline{z}^{b_{n-p}}\nonumber
\end{align}
and
\begin{align*}
\varpi\wedge\ast \overline{\phi}=\frac{1}{p!q!}\varpi_{\ell_1\cdots\ell_p\overline{k_1}\cdots \overline{k_{q}}} \overline{\phi_{i_1\cdots i_p\overline{j_1}\cdots\overline{j_q}}}g^{\overline{i_1}\ell_1}
\cdots g^{\overline{i_p}\ell_p}g^{\overline{k_1}j_1}\cdots g^{\overline{k_q}j_q}\frac{\omega^n}{n!},
\end{align*}
where $\varpi=\frac{1}{p!q!}\varpi_{i_1\cdots i_p\overline{j_1}\cdots\overline{j_q}}\md z^{i_1}\wedge\cdots\wedge\md z^{i_p}\wedge\md \overline{z}^{j_1}\wedge\cdots\wedge\md \overline{z}^{j_q}$ is another $(p,q)$ form and $\det g=\det(g_{i\overline{j}})$. We also define inner product by
$$
\langle\varpi,\,\phi\rangle:=\frac{1}{p!q!}\varpi_{\ell_1\cdots\ell_p\overline{k_1}\cdots \overline{k_{q}}} \overline{\phi_{i_1\cdots i_p\overline{j_1}\cdots\overline{j_q}}}g^{\overline{i_1}\ell_1}
\cdots g^{\overline{i_p}\ell_p}g^{\overline{k_1}j_1}\cdots g^{\overline{k_q}j_q}.
$$
Note that
$$
\ast 1=\frac{\omega^n}{n!},
\quad \overline{\ast \phi}= \ast\overline{ \phi},
$$
where the second equality shows that $\ast$ is a real operator.
From \eqref{relationast}, we can deduce
\begin{align*}
  \ast\ast \phi=(-1)^{p+q}\phi.
\end{align*}
The following basic concepts of positivity can be found in for example \cite[Chapter III]{demaillybook1}.

A $(p,p)$ form $\varphi$ is said to be positive   if for any $(1,0)$ forms $\gamma_j,\,1\leq j\leq n-p$, then
$$
\varphi\wedge\mn\gamma_1\wedge\overline{\gamma_1}\wedge\cdots\wedge\mn \gamma_{n-p}\wedge\overline{\gamma_{n-p}}
$$
is a positive $(n,n)$ form. Any positive $(p,p)$ form $\varphi$ is real, i.e., $\overline{\varphi}=\varphi$. In particular, in the local coordinates, a real $(1,1)$ form
\begin{align}\label{11}
\phi=\mn\phi_{i\overline{j}}\md z^i\wedge\md\overline{z}^j
\end{align}
is positive if and only if $(\phi_{i\overline{j}})$ is a semi-positive Hermitian matrix and we denote $\det \phi:=\det(\phi_{i\overline{j}})$. Similarly, a real $(n-1,n-1)$ form
\begin{align}\label{n-1}
\psi=&(\mn)^{n-1}\sum\limits_{i,j=1}^n(-1)^{\frac{n(n+1)}{2}+i+j+1}\psi^{\overline{j}i}\\
&\md z^1\wedge\cdots\wedge\widehat{\md z^i}\wedge\cdots\wedge\md z^n\wedge \md \overline{z}^1\wedge\cdots\wedge\widehat{\md\overline{z}^j}\wedge\cdots\wedge\md \overline{z}^n\nonumber
\end{align}
is positive if and only if $(\psi^{\overline{j}i})$ is a semi-positive Hermitian matrix and we denote $\det\psi:=\det(\psi^{\overline{j}i})$.
We remark that one can call a real $(1,1)$ form $\phi$ ( resp. a real $(n-1,n-1)$ form $\psi$) strictly positive
if the
Hermitian matrix $(\phi_{i \overline{j}})$ (resp. $(\psi^{\overline{j} i})$) is positive
definite.

For a strictly positive $(1,1)$ form $\phi$ defined as in \eqref{11}, we can deduce a strictly positive $(n-1,n-1)$ form
\begin{align}\label{11n1n1formula}
\frac{\phi^{n-1}}{(n-1)!}=&(\mn)^{n-1}\sum\limits_{k,\ell=1}^n(-1)^{\frac{n(n+1)}{2}+k+\ell+1}\mathrm{det}(\phi_{i\overline{j}})\tilde{\phi}^{\overline{\ell}k}\\
&\md z^{ 1}\wedge\cdots\wedge\widehat{\md z^k}\wedge \cdots\wedge\md z^{n}\wedge\md\overline{z}^{ 1}\wedge\cdots\wedge\widehat{\md \overline{z}^{\ell}}\wedge\cdots\wedge\cdots\wedge \md\overline{z}^{n}\nonumber
\end{align}
where $(\tilde{\phi}^{\overline{\ell}k} )$ is the inverse matrix of $(\phi_{i\overline{j}})$, i.e.,  $\sum\limits_{\ell=1}^n\tilde{\phi}^{\overline{\ell}j}\phi_{k\overline{\ell}}=\delta_{k}^j$. Hence we have
\begin{align}\label{detchin-1}
\det\left(\frac{\phi^{n-1}}{(n-1)!}\right)=\left(\det\phi\right)^{n-1}.
\end{align}
Furthermore, we have
\begin{lem}\label{11n1n1}
Let $(M, g)$ be a complex $n$-dimensional Hermitian manifold. Then there exists a bijection from the space of strictly positive definite $(1,1)$ forms to strictly positive definite $(n-1,n-1)$  forms, given by
\begin{align*}
\phi\mapsto \frac{\phi^{n-1}}{(n-1)!}.
\end{align*}
\end{lem}
The above bijection can be found in \cite{michelsohn} and proved by orthonormal basis. We can also use \eqref{n-1} and \eqref{11n1n1formula} to give the explicit formulae involved (cf. \cite[Formula (3.3)]{phong}).

For a real $(1,1)$ form $\phi$ defined as in \eqref{11} (no need to be positive), \eqref{starformulacomplex} implies
\begin{align}
\ast\phi
=&(\mn)^{n-1}\sum\limits_{k,\ell=1}^n(-1)^{\frac{n(n-1)}{2}+n+k+\ell+1}(\det\omega)\phi^{\overline{\ell}k}\\
&\md z^1\wedge\cdots\wedge\widehat{\md z^k}\wedge\cdots\wedge \md z^n\wedge\md \overline{z}^1\wedge\cdots\wedge\widehat{\md\overline{z}^{\ell}}\wedge\cdots\wedge\md\overline{z}^n,\nonumber
\end{align}
where $\phi^{\overline{\ell}k}=g^{\overline{\ell}i}\phi_{i\overline{j}}g^{\overline{j}k}$. Hence, if $\xi$ is another real $(1,1)$ form with $\det\xi\neq0$, then we can deduce
\begin{align}\label{astdet}
\frac{\det(\ast\phi)}{\det(\ast\xi)}=\frac{\det\phi}{\det\xi}.
\end{align}
We need the following useful formulae   and the proofs are direct and complicated computation.
\begin{lem}[see for example \cite{tw1310}]\label{lemformula1}
For any real $(1,1)$ form $\chi=\mn\chi_{i\overline{j}}\md z^i\wedge\md \overline{z}^j$, we have
\begin{align}\label{formula1}
\ast(\chi\wedge\omega^{n-2})=(n-2)!\left[(\tr_{\omega}\chi)\omega-\chi\right].
\end{align}
\end{lem}
The Christoffel symbols, torsion and curvature of the Chern connection (see for example \cite{twjdg}) are
\begin{align*}
\Gamma_{ij}^k=&g^{\overline{q}k}\partial_ig_{j\overline{q}},\quad T_{ij}^k=\Gamma_{ij}^k-\Gamma_{j i}^k,\\
 R_{i\overline{j}k}{}^{\ell}=&-\partial_{\overline{j}}\Gamma_{ik}^{\ell},\quad R_{i\overline{j}k\overline{\ell}}=R_{i\overline{j}k}{}^p g_{p\overline{\ell}}.
\end{align*}
Denote
\begin{align*}
T_{ij\overline{\ell}}:=T_{ij}^kg_{k\overline{\ell}}=g^{\overline{q}k}\left(\partial_ig_{j\overline{q}}
-\partial_jg_{i\overline{q}}\right)g_{k\overline{\ell}}=\partial_ig_{j\overline{\ell}}
-\partial_jg_{i\overline{\ell}}.
\end{align*}
For a $(1,0)$ form $a=a_{\ell}\md z^{\ell}$, define covariant derivative $\nabla_ia_{\ell}$ by
\begin{align*}
\nabla_{i}a_{\ell}:=\partial_{i}a_{\ell}-\Gamma_{i\ell}^pa_p.
\end{align*}
Then we can deduce
\begin{align}\label{commutate}
[\nabla_{i},\nabla_{\overline{j}}]a_{\ell}=-R_{i\overline{j}\ell}{}^pa_p,\quad [\nabla_{i},\nabla_{\overline{j}}]a_{\overline{m}}=R_{i\overline{j}}{}^{\overline{q}}{}_{\overline{m}}a_{\overline{q}},
\end{align}
where $R_{i\overline{j}}{}^{\overline{q}}{}_{\overline{m}}=R_{i\overline{j}p}{}^{\ell}g^{\overline{q}p}g_{\ell \overline{m}}$.
For any $u\in C^{\infty}(M,\mathbb{R})$, we have
\begin{align}\label{ricciidentityu}
\nabla_{i}u=\partial_{i}u=:u_i,\quad \nabla_{\overline{j}}u=\partial_{\overline{j}}u=:u_{\overline{j}},\quad \nabla_{\overline{j}}\nabla_{i}u=\partial_{\overline{j}}\partial_{i}u=:u_{i\overline{j}},\quad [\nabla_{i},\nabla_j]u=-T_{ij}^pu_p.
\end{align}
Using \eqref{commutate}, we can get the following commutation formulae:
\begin{align}\label{ricciidentity}
\nabla_{\ell}u_{i\overline{j}}=&\nabla_{\overline{j}}\nabla_{\ell}u_i-R_{\ell\overline{j}i}{}^pu_p,\; \nabla_{\overline{m}}u_{p\overline{j}}=\nabla_{\overline{j}}u_{p\overline{m}}-\overline{T_{mj}^q}u_{p\overline{q}},
\;\nabla_{\ell}u_{i\overline{q}}=\nabla_{i}u_{\ell\overline{q}}-T_{\ell i}^pu_{p\overline{q}},\\
\nabla_{\overline{m}}\nabla_{\ell}u_{i\overline{j}}=&\nabla_{\overline{j}}\nabla_{i}u_{\ell\overline{m}}
+R_{\ell\overline{m}i}{}^pu_{p\overline{j}}-R_{i\overline{j}\ell}{}^pu_{p\overline{m}}-T_{\ell i}^p\nabla_{\overline{j}}u_{p\overline{m}}-\overline{T_{mj}^{q}}\nabla_{i}u_{\ell\overline{q}}-T_{i\ell}^p\overline{T_{mj}^q}u_{p\overline{q}}.
\nonumber
\end{align}
\begin{lem}\label{lemformula2}
For any $u\in C^{\infty}(M,\mathbb{R})$, we have
\begin{align}\label{formula2}
\ast (\partial u\wedge \dbar(\omega^{n-2}))=(n-2)!\left(u_p\overline{T_{\ell q}^{\ell}}-\overline{T_{\ell j}^{\ell}}g^{\overline{j}i}u_i g_{p\overline{q}}-g_{p\overline{\ell}}g^{\overline{\ell_1}k}u_k\overline{T_{\ell_1 q}^{\ell}}\right)\md z^p\wedge\md \overline{z}^q
\end{align}
\end{lem}
Note that $\ast$ is a real operator. From Lemma \ref{lemformula2}, we can deduce the formula in \cite{stw1503} as follows.
\begin{cor}
For any $u\in C^{\infty}(M,\mathbb{R})$, we have
\begin{align}\label{zuintro}
Z(u):=\frac{1}{(n-1)!}\ast \left[\mathrm{Re}\left(\mn\partial u\wedge \dbar(\omega^{n-2})\right)\right]=\mn Z_{p\overline{q}}\md z^p\wedge\md \overline{z}^q,
\end{align}
where
\begin{align}\label{zupq}
Z_{p\overline{q}}=&\frac{1}{2(n-1)}\left[u_p\overline{T_{\ell q}^{\ell}}-\overline{T_{\ell j}^{\ell}}g^{\overline{j}i}u_i g_{p\overline{q}}-g_{p\overline{\ell}}g^{\overline{\ell_1}k}u_k\overline{T_{\ell_1 q}^{\ell}}\right.\\
&\quad\quad\quad\quad\left.+ u_{\overline{q}} T_{\ell p}^{\ell} -T_{\ell j}^{\ell}g^{\overline{i}j}u_{\overline{i}} g_{p\overline{q}}-g_{\ell\overline{q}}g^{\overline{k}\ell_1}u_{\overline{k}} T_{\ell_1 p}^{\ell}  \right].\nonumber
\end{align}
We can also write
$$
Z_{p\overline{q}}=Z^{i}_{p\overline{q}}u_i+\overline{ Z^{i}_{q\overline{p}}}u_{\overline{i}},
$$
where
\begin{align}\label{zipq}
Z^{i}_{p\overline{q}}=\frac{1}{2(n-1)}\left(\delta_{p}^i\overline{T_{\ell q}^{\ell}}-\overline{T_{\ell j}^{\ell}}g^{\overline{j}i}  g_{p\overline{q}}-g_{p\overline{\ell}}g^{\overline{k}i} \overline{T_{k q}^{\ell}}\right).
\end{align}
\end{cor}
Note that $Z(u)$ is linear in $\nabla u$. The following useful lemma is simple and we will use it without  pointing it out again and again (cf.\cite{gill1410}).
\begin{lem}\label{lemzu}
For any $f\in C^{\infty}(M,\mathbb{R})$, at the point  where $\ddbar f\leq(\geq)0$, we have
\begin{align*}
(\Delta_g f)\omega-\ddbar f\leq(\geq)0,
\end{align*}
where $\Delta_g f=g^{\overline{j}i}\partial_i\partial_{\overline{j}}f$. At the point where $\nabla f=0$, we have $Z(f)=0$.
\end{lem}
To end this section, we introduce some terminology concerning cohomology classes of $(n-1,n-1)$ forms. Define the \emph{Aeppli cohomology group} (see \cite{tw1310})
$$
H_A^{n-1,n-1}(M,\mathbb{R}):=\frac{\{\partial\dbar{\textrm-}\text{closed real }(n-1,n-1)\text{ forms}\}}{\{\partial\gamma+\overline{\partial \gamma}:\;\gamma \,\in\,\Lambda^{n-2,n-1}(M)\}}.
$$
This space is naturally in duality with the finite dimensional \emph{Bott-Chern cohomology group} with the nondegenerated pairing
$$
H_A^{n-1,n-1}(M,\mathbb{R})\otimes H_{\mathrm{BC}}^{1,1}(M,\mathbb{R})\longrightarrow \mathbb{R}
$$
given by wedge product and integration over on $M$ (see \cite{angella}), where
$$H^{1,1}_{\mathrm{BC}}(M,\mathbb{R})=\frac{\{\md{\textrm-}\text{closed real } (1,1)\text{ forms}\}}{\{\ddbar\psi: \; \psi\in C^\infty(M,\mathbb{R})\}}.$$
For any $u\in C^{\infty}(M,\mathbb{R})$, define
$$
\gamma:=\frac{\mn}{2}\dbar u\wedge\chi^{n-2},
$$
where $\chi$ is a real $(1,1)$ form. Then we have
\begin{align}\label{betau}
\beta_{u}:=\partial\gamma+\overline{\partial \gamma}=\ddbar u\wedge \chi^{n-2}+\mathrm{Re}\left(\mn\partial u\wedge\dbar(\chi^{n-2})\right).
\end{align}
$\beta_u$ is $\partial\dbar$\text{-} closed. Indeed, it is the $(n-1,n-1)$ part of the $\md$-exact $(2n-2)$ form $\md \left(\md^cu\wedge\chi^{n-2}\right)$, where $$\md^c=\frac{\mn}{2}(\dbar-\partial)$$ with $\md\md^c=\ddbar$.
Let $\alpha$ and $\alpha'$ be Hermitian metrics on $M$. Then
\begin{align*}
\ddbar \left(\log\frac{\alpha^n}{\alpha'^n}\right)\wedge \chi^{n-2}+\mathrm{Re}\left[\mn\partial\left(\log\frac{\alpha^n}{\alpha'^n}\right)\wedge\dbar(\chi^{n-2})\right],
\end{align*}
is well-defined  $\partial\dbar$-closed since
$$\log\frac{\alpha^n}{\alpha'^n}\in C^{\infty}(M,\mathbb{R}).$$

\section{Preliminary estimates}\label{secpreliminary}
Define a linear operator
\begin{equation}
\label{Ldefn}
L(\varphi):=\Theta^{\overline{j}i}\partial_i\partial_{\overline{j}}\varphi+\tilde g^{\overline{j}i}Z(\varphi)_{i\overline{j}}
=\Theta^{\overline{j}i}\partial_i\partial_{\overline{j}}\varphi+\tr_{\tilde\omega}Z(\varphi)
\end{equation}
with
\begin{align*}
\Theta^{\overline{j}i}=\frac{1}{n-1}\left(\left(\tr_{\tilde \omega}\alpha\right)\alpha^{\overline{j}i}-\tilde g^{\overline{j}i}\right)>0.
\end{align*}
Obviously, $L$ is a second order elliptic operator. Noting that $L$ is the linearized operator of \eqref{paragau}, standard parabolic theory implies that there exists a smooth solution $u$ to \eqref{paragau} on $[0,T)$, where $[0,\,T)$ is the maximal time interval with $T\in (0,\,\infty]$. We will prove $T=\infty$. First, we give a preliminary estimate as follows.
\begin{lem}\label{lempreliminary}
Let   $u$ be the solution to \eqref{paragau} on $M\times[0,\,T)$. Then there exists a uniform constant $C$, i.e., depending only on the initial data on $M$, such that
\begin{align}
\sup\limits_{M\times[0,\,T)}\left|\frac{\partial u}{\partial t}(x,t)\right|\leq C.
\end{align}
\end{lem}
\begin{proof}
From \eqref{paragau}, we get the evolution equation for $\dot{u}:=\ppt u$
\begin{equation}\label{ttu}
\ppt\dot{u}=L(\dot{u}).
\end{equation}
By the maximum principle, we get
\begin{align*}
\sup\limits_{M\times[0,\,T)}\left|\frac{\partial u}{\partial t}(x,t)\right|
\leq&\sup_M\left|\frac{\partial u}{\partial t}(x,0)\right|\\
\leq&\left\|\log\frac{\left(\varpi+\frac{1}{n-1}\left[(\Delta u_0)\alpha-\ddbar u_0\right]+Z(u_0)\right)^n}{\alpha^n}\right\|_{L^{\infty}(M)}+\|\psi\|_{L^{\infty}(M)},
\end{align*}
as required.
\end{proof}
Next, using Lemma \ref{lempreliminary}, we can get the estimate of $\tilde u$.
\begin{prop}
Let   $u$ be the solution to \eqref{paragau} on $M\times[0,\,T)$. Then there exists a uniform constant $C$
such that
\begin{align}
\sup\limits_{M\times[0,\,T)}\left|\tilde u(x,t)\right|\leq C.
\end{align}
\end{prop}
\begin{proof}
We can rewrite \eqref{paragau} as
\begin{equation}
\log\frac{\left(\varpi+\frac{1}{n-1}\left[(\Delta u)\alpha-\ddbar u\right]+Z(u)\right)^n}{\alpha^n}=\psi+\dot u.
\end{equation}
Then by \cite[Theorem 1.6]{tw1310} (see also \cite[Remark 12]{sz}), there exists a constant $C'$ depending only on the initial data on $M$ and $\sup_M|\psi+\dot u|$, such that
\begin{equation}
\label{reftw1310thm1.6}
\sup_M\left|u(x,t)-u(y,t)\right|\leq C'
\end{equation}
By Lemma \ref{lempreliminary}, if follows that $\sup_M|\psi+\dot u|$ is uniformly bounded, and hence $C'$ in\eqref{reftw1310thm1.6} is a  uniform constant. Since we have $\int_M\tilde u\alpha^n=0$ by definition, there exists a point $(y,t)$ such that $\tilde u(y,t)=0$ and hence for any $(x,t)\in M\times[0,T)$, we get
\begin{equation}
\left|\tilde u(x,t)\right|=\left|\tilde u(x,t)-\tilde u(y,t)\right|=\left| u(x,t)- u(y,t)\right|\leq C,
\end{equation}
as required.
\end{proof}
\section{Second order estimate}\label{secsecond}
We can use the ideas from \cite{stw1503} in the elliptic setting to prove the second estimate.
\begin{thm}\label{2ndestimate}
Let   $u$ be the solution to \eqref{paragau} on $M\times[0,\,T)$. Then there exists a uniform $C>0$ such that
\begin{align}\label{formula2ndestimate}
\sup\limits_{M\times[0,\,T)}|\ddbar u|_{\alpha}\leq CK,
\end{align}
where $$K=1+\sup\limits_{M\times[0,\,T)}|\nabla u|_{\alpha}^2.$$
\end{thm}
We need some preliminaries. For any real $(1,1)$ form $\xi$, we define
$$
P_{\alpha}(\xi):=\frac{1}{n-1}\Big((\tr_{\alpha}\xi)\alpha-\xi\Big)=\frac{1}{(n-1)!}\ast(\xi\wedge\alpha^{n-2}).
$$
Note that $\tr_{\alpha}\xi=\tr_{\alpha}\left(P_{\alpha}(\xi)\right)$ and
$$
\xi=\left(\tr_{\alpha}\left(P_{\alpha}(\xi)\right)\right)\alpha-(n-1)P_{\alpha}(\xi).
$$
Denote
\begin{align*}
\chi_{i\overline{j}}=&(\tr_{\alpha}\varpi)\alpha_{i\overline{j}}-(n-1)\varpi_{i\overline{j}},
\end{align*}
with $P_{\alpha}(\chi)=\varpi$, and
$$
W(u)_{i\overline{j}}=\left(\tr_{\alpha}Z(u)\right)\alpha_{i\overline{j}}-(n-1)Z(u)_{i\overline{j}}
=:W_{i\overline{j}}^pu_p+\overline{W^p_{j\overline{i}}}u_{\overline{p}}.$$
Then we define
\begin{align}\label{2ndgij}
g_{i\overline{j}}
=\chi_{i\overline{j}}+u_{i\overline{j}}+W(u)_{i\overline{j}}
\end{align}
with $P_{\alpha}(g_{i\overline{j}})=\tilde g_{i\overline{j}}$.

In orthonormal coordinates for $\alpha$ at any given point,  it follows that the component $Z_{i\overline{j}}$ is independent of $u_{\overline{i}}$ and $u_{j}$, and that $\nabla_{i} Z_{i\overline{i}}$ is independent of $u_i,u_{\overline{i}}, u_{i\overline{i}},\nabla_{i}u_i $. Indeed, in such local coordinates, from \eqref{zupq}, we have
\begin{align*}
Z_{i \overline{i}}
=&\frac{1}{2(n-1)}\left(\sum\limits_{p\neq i}\sum\limits_{k\neq i}u_p\overline{T_{pk}^k}+\sum\limits_{p\neq i}\sum\limits_{k\neq i}u_{\overline{p}}T_{pk}^k\right),\\
\nabla_{i}Z_{i \overline{i}}
=&\frac{1}{2(n-1)}\left(\sum\limits_{p\neq i}\sum\limits_{k\neq i}\left(\nabla_iu_{p}\overline{T_{pk}^k}+u_{p}\overline{\nabla_iT_{pk}^k}\right)+\sum\limits_{p\neq i}\sum\limits_{k\neq i}\left(u_{i\overline{p}}T_{pk}^k+u_{\overline{p}}\nabla_iT_{pk}^k\right)\right),\\
Z_{i\overline{j}}
=&\frac{1}{2(n-1)}\left(-\sum\limits\left(u_i\overline{T_{jk}^k}+u_k\overline{T_{kj\overline{i}}}\right)-\sum\limits_{k\neq i}\left(u_{\overline{j}}T_{ik}^k+u_{\overline{k}}T_{ki\overline{j}}\right)\right),
\end{align*}
where $i\neq j$ are fixed indices and in the last equality, we use the skew-symmetry of the torsion, as desired.

Furthermore, $\nabla_i\nabla_{\overline{i}} Z_{i \overline{i}}$ is independent of $\nabla_iu_i,\nabla_{\overline{i}}u_{\overline{i}},\nabla_{\overline{i}}\nabla_iu_i$ and $\nabla_{\overline{i}}u_{i\overline{i}}$. For any index $p$, $\nabla_iZ_{p\overline{i}}$ is independent of $\nabla_{i}u_i$. In this following part of this section, we will use these properties directly and do not prove them again.

Denote $\left(\tilde B_{i}{}^j\right)=\left(\tilde g_{i\overline{\ell}}\alpha^{\overline{\ell}j}\right)$ which can be seen as the endomorphism of $T^{1,0}M$. This endomorphism is Hermitian with respect to the Hermitian metric $\alpha$, i.e., for any tangent vectors $X=X^{i}\partial_i$ and $Y=Y^j\partial_j$, we have
\begin{align*}
\langle \tilde BX ,\,Y\rangle_{\alpha}=\tilde B_{i}{}^kX^i\alpha_{k\overline{\ell}}\overline{Y^{\ell}}
=\tilde g_{i\overline{q}}\alpha^{\overline{q}k}X^i\alpha_{k\overline{\ell}}\overline{Y^{\ell}}=\tilde g_{i\overline{\ell}}X^i \overline{Y^{\ell}}=X^i\alpha_{i\overline{j}}\overline{\tilde B_{p}{}^{j}Y^p}=\langle X ,\,\tilde BY\rangle_{\alpha}.
\end{align*}
We define
$$
\tilde F(\tilde B)=\log \det \tilde B=\log(\mu_1\cdots\mu_n)=:\tilde f(\mu_1,\cdots,\mu_n),
$$
where $\mu_1,\cdots,\mu_n$ are the eigenvalues of $\left(\tilde B_{i}{}^j\right)$. Then \eqref{paragau} can be rewritten as
\begin{align}\label{udef2nd1}
\tilde F (\tilde B)=\dot{u}+\psi=:h.
\end{align}
For   $\tilde f$ and $h$,  we have
\begin{enumerate}
\item[(i)] $\tilde f$ is defined on
$$
\tilde\Gamma:=\Gamma_n=\Big\{(x_1,\cdots,x_n)\in \mathbb{R}^n:\;x_i>0,\;i=1,\cdots,n\Big\}.
$$
\item[(ii)] $\tilde f$ is symmetric, smooth, concave and increasing, i.e., $\tilde f_i>0$ for all $i$.
\item[(iii)] $\sup_{\partial\tilde\Gamma}\tilde f\leq \inf_{M\times[0,T)}h$.
\item[(iv)] For any $\mu \in \tilde\Gamma$, we get $\lim_{t\rightarrow\infty}\tilde f(t\mu)=\sup_{\tilde\Gamma}\tilde f=\infty$.
\item[(v)]$h$ is bounded on $M\times[0,T)$ thanks to the estimate of $|\dot{u}|$ in Lemma \ref{lempreliminary}.
\end{enumerate}
We also define
$$
F (A):=\tilde F(\tilde B)=:f(\lambda_1,\cdots,\lambda_n),
$$
where $(A_{i}{}^{j})=\left(g_{i\overline{\ell}}\alpha^{\overline{\ell}j}\right)$, which is also an endomorphism of $T^{1,0}M$   with respect to the Hermitian metric $\alpha$, and $\lambda_1,\cdots,\lambda_n$ are its eigenvalues. There exists a map
\begin{align}
\label{Pdefn}
P:\;\mathbb{R}^n\longrightarrow \mathbb{R}^n,\quad \mu_k=\frac{1}{n-1}\sum\limits_{i\neq k}\lambda_{i},
\end{align}
induced   by $P_{\alpha}$ above. Then we have
$$
f(\lambda_1,\cdots,\lambda_n)=\tilde f\circ P(\lambda_1,\cdots,\lambda_n)
$$
defined on $\Gamma:=P^{-1}(\tilde\Gamma)$. Clearly,  $f$  satisfies the same conditions as $\tilde f$. Then \eqref{udef2nd1} can also be rewritten as
\begin{align}\label{udef2nd2}
F(A)=h.
\end{align}
We make some simple calculation about $\tilde f$ and $f$.
Since
\begin{align}\label{daotildef}
\tilde f_i=\frac{1}{\mu_i},
\end{align}
we can get
\begin{align}\label{daof}
f_i=\frac{1}{n-1}\sum\limits_{k\neq i}\frac{1}{\mu_k}.
\end{align}
Suppose that $\lambda_1\geq \lambda_2\geq\cdots\geq \lambda_n\in \Gamma$. From the definition of $P$, \eqref{daotildef} and \eqref{daof}, we have
\begin{align}
0 <\mu_1\leq \mu_2\leq \cdots\leq \mu_n,\nonumber\\
\tilde f_1\geq \tilde f_2\geq \cdots\geq \tilde f_n,\nonumber\\
0<f_1\leq f_2\leq\cdots\leq f_n,\nonumber\\
\label{daoshueigenn}\sum\limits_{k=1}^n\lambda_kf_k=\sum\limits_{k=1}^n\mu_k\tilde f_k=n.
\end{align}
where we also use the fact that $\left(\tilde B_i{}^j\right)$ is positive definite. For $k\geq 2$, we have
\begin{align}\label{tildef1fkf1}
  0<\frac{\tilde f_1}{n-1}\leq f_k\leq \tilde f_1.
\end{align}
and
\begin{align}\label{tildefk}
\tilde f_k\leq (n-1)f_1.
\end{align}
\begin{prop}\label{prop}
For any $x\in M$, choose orthonormal coordinates for $\alpha$ at $x$, with $g$ defined as in \eqref{2ndgij} is diagonal with eigenvalues $(\lambda_1,\cdots,\lambda_n)$. Then we have
$$
|\lambda|\geq R,
$$
for a uniform constant $R>0$, and for $f(\lambda)=\dot{u}+\psi$, there also holds two possibilities as follows.
\begin{enumerate}
\item[(a)]  We have
\begin{align}\label{partaformula}
\sum\limits_kf_{k}(\lambda)(\chi_{k\overline{k}}-\lambda_k)-|\dot{u}|>\kappa \sum\limits_kf_k(\lambda).
\end{align}
\item[(b)]  Or we have
\begin{align}
f_k(\lambda)>\kappa\sum\limits_{i=1}^nf_{i}(\lambda)
\end{align}
for all $k=1,2,\cdots,n$.
\end{enumerate}
In addition, we have $\sum\limits_{k=1}^nf_{k}(\lambda)>\kappa$. Here $0<\kappa<1$ is a uniform constant.
\end{prop}
\begin{proof}
By the Cauchy-Schwarz inequality and the inequality of arithmetic and geometric means, we get
$$
|\lambda|^2\geq \frac{(\lambda_1+\cdots+\lambda_n)^2}{n}=\frac{(\mu_1+\cdots+\mu_n)^2}{n}\geq n(\mu_1\cdots\mu_n)^{\frac{2}{n}}=ne^{2h/n} \geq R^2,
$$
where for the last inequality we use the fact that $\tilde F(\tilde B)=\dot u+\psi$ is uniformly bounded by Lemma \ref{lempreliminary}.

Since $\chi_{k\overline{k}}=\left(\sum\limits_{i=1}^n\varpi_{i\overline{i}}\right)-(n-1)\varpi_{k\overline{k}}$, \eqref{daof} implies
\begin{align}\label{parta1}
\sum\limits_{k}f_{k}(\lambda)\chi_{k\overline{k}}=\sum\limits_{k}\tilde f_{k}(\lambda)\varpi_{k\overline{k}}=\sum\limits_{k}\frac{1}{\mu_k} \varpi_{k\overline{k}}>\tau\sum\limits_{k}\frac{1}{\mu_k}=\tau\sum\limits_{k}f_k(\lambda),
\end{align}
where we use the fact that $\varpi$ has a uniform lower bound. If $\lambda_1$ is relatively large, i.e., $\mu_1$ is relatively small such that $\frac{\tau}{2\mu_1}>n+C\geq\sum\limits_{k=1}^n\lambda_kf_k(\lambda)+|\dot{u}|$, where $C$ is the uniform bound of $\dot{u}$ in Lemma \ref{lempreliminary}, then we have
\begin{align*}
\sum\limits_{k}f_{k}(\lambda)\chi_{k\overline{k}}
\geq& \frac{\tau}{2}\sum\limits_{k}\frac{1}{\mu_k}+\frac{\tau}{2}\sum\limits_{k}\frac{1}{\mu_k}\\
\geq &\frac{\tau}{2}\sum\limits_{k}\frac{1}{\mu_k}+\frac{\tau}{2\mu_1}\\
\geq &\frac{\tau}{2}\sum\limits_{k}\frac{1}{\mu_k}+n+C\\
\geq&\frac{\tau}{2}\sum\limits_{k}f_k(\lambda)+\sum\limits_{k=1}^n\lambda_kf_k(\lambda)+|\dot u|,
\end{align*}
as required. Otherwise, there exists a large constant $A>1$ such that  $\mu_2\leq\cdots\leq\mu_n\leq A\mu_1$ since $\dot{u}+\psi=\log(\mu_1\cdots\mu_n)$ is bounded and we just need to prove
\begin{align*}
f_n\geq\cdots\geq f_1=&\frac{1}{n-1}\sum\limits_{i=2}^n\frac{1}{\mu_i}\geq \kappa \sum\limits_{i=1}^n\frac{1}{\mu_i}.
\end{align*}
To get this, we need to choose $0<\kappa<1$ such that
$$
(1-\kappa)\sum\limits_{i=2}^n\frac{1}{\mu_i}\geq \frac{(1-\kappa)(n-1)}{A\mu_1}\geq \frac{\kappa}{\mu_1},
$$
as required.

In addition, using the inequality of arithmetic and geometric means, we can get
\begin{align*}
\sum\limits_{k=1}^n f_k(\lambda)
=&\sum\limits_{k=1}^n\frac{1}{\mu_k}\geq n(\mu_1\cdots\mu_n)^{-1/n}=ne^{-h/n}\geq \kappa,
\end{align*}
where we use the fact that $h$ is uniformly bounded by Lemma \ref{lempreliminary}.
\end{proof}
Now we need some basic formulae for the derivatives of eigenvalues (see for example \cite{spruck}).
\begin{lem}[Spruck \cite{spruck}]
The derivative of the eigenvalue $\lambda_i$ at a diagonal matrix $(A_{i}{}^j)$ with distinct eigenvalue are
\begin{align}
\label{eigenvalue1st}\lambda_{i}^{pq}=&\delta_{pi}\delta_{qi},\\
\label{eigenvalue2nd}\lambda_{i}^{pq,rs}=&(1-\delta_{ip})\frac{\delta_{iq}\delta_{ir}\delta_{ps}}{\lambda_i-\lambda_p}
+(1-\delta_{ir})\frac{\delta_{is}\delta_{ip}\delta_{rq}}{\lambda_i-\lambda_r},
\end{align}
where
$$
\lambda_{i}^{pq}=\frac{\partial\lambda_i}{\partial  A_{p}{}^q}, \quad
\lambda_{i}^{pq,rs}=\frac{\partial^2\lambda_i}{\partial  A_{p}{}^q\partial  A_{r}{}^s}.
$$
\end{lem}
\begin{lem}[Gerhardt \cite{gerhardt}]
\label{lemgerhardt}
If $F(A)=f(\lambda_1,\cdots, \lambda_n)$ in terms of a symmetric funtion of the eigenvalues, then at a diagonal matrix $(A_{i}{}^j)$ with distinct eigenvalues we have
\begin{align}
\label{f1stdaoshu}F^{ij}=&\delta_{ij} f_i,\\
\label{f2nddaoshu}F^{ij,rs}=&f_{ir}\delta_{ij}\delta_{rs}+\frac{f_i-f_j}{\lambda_i-\lambda_j}(1-\delta_{ij})\delta_{is}\delta_{jr},
\end{align}
where
$$
F^{ij}=\frac{\partial F}{\partial  A_{i}{}^j}, \quad F^{pq,rs}=\frac{\partial^2F}{\partial  A_{i}{}^j\partial  A_{r}{}^s}.
$$
\end{lem}
These formulae make sense even if the eigenvalues are not distinct, since if $f$ is symmetric  then $F$ is a smooth function on the space of matrices. In particular, we have $f_i\longrightarrow f_j$ as $\lambda_i\longrightarrow \lambda_j$. If $f$ is concave and symmetric, then we have (see \cite{spruck}) that $\frac{f_i- f_j}{\lambda_i-\lambda_j}\leq 0$. In particular, if $\lambda_i\leq \lambda_j$, then we have $f_i\geq f_j$. Also it follows that
\begin{align}\label{cabor67}
F^{ij,rs}\left(\nabla_{k}u_{i\overline{j}}\right)\left(\nabla_{\overline{k}}u_{r\overline{s}}\right)
=&f_{ij}\left(\nabla_{k}u_{i\overline{i}}\right)\left(\nabla_{\overline{k}}u_{j\overline{j}}\right)
+\sum\limits_{p\neq q}\frac{f_p- f_q}{\lambda_p-\lambda_q}|\nabla_{k}u_{p\overline{q}}|^2\\
\leq&f_{ij}\left(\nabla_{k}u_{i\overline{i}}\right)\left(\nabla_{\overline{k}}u_{j\overline{j}}\right)
+\sum\limits_{p>1}\frac{f_1- f_p}{\lambda_1-\lambda_p}|\nabla_{k}u_{p\overline{1}}|^2.\nonumber
\end{align}
%
\begin{proof}[Proof of Theorem \ref{2ndestimate}]
We use the ideas from \cite{stw1503} in the elliptic setting. Since here we consider the parabolic equation, there are some new terms involved time $t$ which we need to estimate and we need Proposition \ref{prop} which is slightly different from \cite[Proposition 2.3]{stw1503}. Also note that here we use covariant derivatives instead of partial derivatives used in \cite{stw1503}\footnote{The author thanks Professor Valentino Tosatti suggesting him this point.}. Hence we can be brief and point out the main differences in the following part.
It is sufficient to prove
$$\lambda_1\leq CK.$$
Indeed, since $\sum\limits_{i=1}^n\lambda_i=\sum\limits_{i=1}^n\mu_i>0$, if $\lambda_1\leq CK$ then so is $|\lambda_k|,\,k=2,\cdots,n$, which implies \eqref{formula2ndestimate}.
To obtain this, we consider the function
\begin{align*}
H=\log\lambda_1+\phi\left(|\nabla u|^2_{\alpha}\right)+\varphi(u),
\end{align*}
where
$$
\phi(s)=-\frac{1}{2}\log\left(1-\frac{s}{2K}\right),\quad\varphi(s)=D_1e^{-D_2s},
$$
with sufficiently large uniform constants $D_1,D_2>0$ to be determined later.
Note that
$$
\phi\left(|\nabla u|^2_{\alpha}\right)\in[0,\,2\log 2]
$$
and
\begin{align*}
\frac{1}{4K}<\phi'<\frac{1}{2K},\quad \phi''=2(\phi')^2.
\end{align*}
We work at a point $(x_0,t_0)$ where $H$ achieves its maximum. Without loss of generality, we assume $0<t_0<T$. Choose orthonormal complex coordinates such that $x_0$ is the origin, $\alpha$ is a unit matrix and $g$ is diagonal with $\lambda_1=g_{1\overline{1}}$. To make sure that $H$ is smooth at this point, we fix a diagonal matrix $B$ with $B_{1}{}^1=0,\; 0<B_{2}{}^2<\cdots<B_{n}{}^n$, and define
$\tilde A=A-B$ with eigenvalues denoted by $\tilde\lambda_1,\cdots,\tilde\lambda_{n}$. Clearly, at this point  $(x_0,t_0)$, we have
$$
\tilde\lambda_1=\lambda_1,\quad \tilde\lambda_i=\lambda_i-B_{i}{}^{i},\quad i=2,\cdots,n
$$
and
$
\tilde\lambda_1>\cdots>\tilde\lambda_n.
$
Noting that
$
\sum\limits_{i=1}^n\lambda_i=\sum\limits_{i=1}^n\mu_i>0,
$
we can assume that  $B$ is small enough such that
$$
\sum\limits_{i=1}^n\tilde\lambda_i>-1
$$
and
\begin{align}\label{6.18}
  \sum\limits_{p>1}\frac{1}{\lambda_1-\tilde\lambda_p}\leq C
\end{align}
where $C$ is a fixed constant depending only on $n$. We give some remarks about $B$. It can also be considered as an endomorphism of $T^{1,0}M$, and is represented by a constant diagonal matrix $(B_{i}{}^j)$ in these local coordinates. Hence, we can deduce
\begin{align*}
\nabla_{\overline{j}}B_{r}{}^s=&\partial_{\overline{j}}B_{r}{}^s=0,\quad \nabla_{i}\nabla_{\overline{j}}B_{r}{}^s=0,\\ \nabla_{i}B_{r}{}^s=&\partial_{i}B_{r}{}^s-\Gamma_{ir}^pB_{p}{}^s+\Gamma_{ip}^sB_{r}{}^p=\Gamma_{ir}^s\left(B_{r}{}^r-B_{s}{}^s\right),\\
\nabla_{\overline{j}}\nabla_{i}B_{r}{}^s=&\partial_{\overline{j}}\nabla_{i}B_{r}{}^s=R_{i\overline{j}r}{}^s\left(B_{s}{}^s-B_{r}{}^r\right).\\
\end{align*}
Now consider the quantity
$$
\tilde H=\log\tilde\lambda_1+\phi\left(|\nabla u|^2_{\alpha}\right)+\varphi(u),
$$
which is smooth in this chart and attains its maximum at the point $(x_0,t_0)$. The following calculation is at this point. We may assume $\lambda_1\gg K>1$. We use subscripts $k$ and $\overline{\ell}$ to denote the partial derivatives $\partial/\partial z^k$ and $\partial/\partial \overline{z}^{\ell}$. As in \cite{stw1503}, we have
\begin{align}
\label{hq}\tilde H_q=&\frac{\tilde{\lambda}_{1,q}}{\lambda_1} + \phi'V_q  +
\varphi' u_q=0, \quad \textrm{for } V_q : = u_ru_{\overline{r}q} +u_{\overline{r}}\nabla_qu_{r}.\\
\label{hqq}\tilde{H}_{q\overline{q}}=& \frac{\tilde{\lambda}_{1,q\overline{q}}}{\lambda_1} -
\frac{|\tilde{\lambda}_{1,q}|^2}{\lambda_1^2}
+ \phi'\Big( u_r\nabla_{\overline{q}}u_{\overline{r}q}+u_{\overline{r}}\nabla_{\overline{q}}\nabla_{q}u_r  +
|\nabla_{q}u_{r}|^2 + |u_{\overline{r}q}|^2 \Big)  \\
& + \phi''|V_q|^2 + \varphi''|u_q|^2 + \varphi' u_{q\overline{q}}. \nonumber
\end{align}
From \eqref{eigenvalue1st}, we get
\begin{align}\label{lambda1p}
\tilde\lambda_{1,p}
= \tilde\lambda_{1}^{rs}\left(\alpha^{\overline{j}s}\nabla_{p}g_{r\overline{j}}-\nabla_pB_{r}{}^s\right)
=\nabla_{p}g_{1\overline{1}}-\nabla_pB_{1}{}^1=\nabla_{p}g_{1\overline{1}},
\end{align}
where we use the fact that $\nabla_pB_{1}{}^1=0$. Then using this formula and \eqref{eigenvalue1st} and \eqref{eigenvalue2nd}, we can deduce
\begin{align*}
\tilde\lambda_{1,p\overline{q}}
=&\tilde\lambda_{1}^{rs,ab}\left(\alpha^{\overline{j}s} \nabla_{p}g_{r\overline{j}}-\nabla_{p}B_{r}{}^s\right) \left(\alpha^{\overline{\ell}b} \nabla_{\overline{q}}g_{a\overline{\ell}}-\nabla_{\overline{q}}B_{a}{}^b\right)
+\tilde\lambda_{1}^{rs}\left(\alpha^{\overline{j}s}\nabla_{\overline{q}}\nabla_{p}g_{r\overline{j}}-\nabla_{\overline{q}}\nabla_{p}B_{r}{}^s \right)  \\
=&\tilde\lambda_{1}^{rs,ab}\left(\nabla_{p}g_{r\overline{s}}-\nabla_{p}B_{r}{}^s\right) \left(\nabla_{\overline{q}}g_{a\overline{b}}\right)
+\tilde\lambda_{1}^{rs}\left(\nabla_{\overline{q}}\nabla_{p}g_{r\overline{s}}-\nabla_{\overline{q}}\nabla_{p}B_{r}{}^s \right)\nonumber  \\
=&\sum\limits_{r\neq 1}\frac{1}{\lambda_1-\tilde\lambda_{r}}\left(\nabla_{p}g_{r\overline{1}} -\nabla_{p}B_{r}{}^1\right)\left(\nabla_{\overline{q}}g_{1\overline{r}}\right)
+\sum\limits_{a\neq 1}\frac{1}{\lambda_1-\tilde\lambda_{a}}\left(\nabla_{p}g_{1\overline{a}}-\nabla_{p}B_{1}{}^a\right)\left(\nabla_{\overline{q}}g_{a\overline{1}}\right) \nonumber\\
&+\nabla_{\overline{q}}\nabla_{p}g_{1\overline{1}},      \nonumber
\end{align*}
where we also use the fact that $\nabla_{\overline{q}}B_{a}{}^b=\nabla_{\overline{q}}\nabla_{p}B_{1}{}^1=0$. In particular, we have
\begin{align}\label{lambda1qq}
\tilde\lambda_{1,q\overline{q}}=&\nabla_{\overline{q}}\nabla_{q}g_{1\overline{1}}+ \sum\limits_{r> 1}\frac{|\nabla_{q}g_{r\overline{1}} |^2+|\nabla_{q}g_{1\overline{r}}|^2}{\lambda_1-\tilde\lambda_{r}}
 -\sum\limits_{r> 1}\frac{\left(\nabla_{q}B_{r}{}^1\right)\left(\nabla_{\overline{q}}g_{1\overline{r}}\right)
 +\left(\nabla_{q}B_{1}{}^r\right)\left(\nabla_{\overline{q}}g_{r\overline{1}}\right)}{\lambda_1-\tilde\lambda_{r}}.
\end{align}
From \eqref{udef2nd2} and \eqref{f1stdaoshu}, we have
\begin{align}\label{dotuk}
\dot{u}_k
= F^{rs}\nabla_k\left(g_{r\overline{j}}\alpha^{\overline{j}s} \right) -\psi_k
= F^{rs}\alpha^{\overline{j}s}\nabla_kg_{r\overline{j}} -\psi_k
= F^{rr}\nabla_kg_{r\overline{r}} -\psi_k .
\end{align}
Also from \eqref{f1stdaoshu}, we know that $F^{kk}=f_k,\,\tilde F^{kk}=\tilde f_k$. For later use, we write $\mathcal{F}=\sum\limits_{k=1}^n F^{kk}\geq \kappa$.
The formula \eqref{dotuk} together with \eqref{f1stdaoshu} and \eqref{f2nddaoshu} implies
\begin{align}\label{dotukl}
\dot{u}_{k\overline{\ell}}
=&F^{rs,ab}\nabla_{k}\left(g_{r\overline{j}}\alpha^{\overline{j}s} \right)\nabla_{\overline{\ell}}\left(g_{a\overline{c}}\alpha^{\overline{c}b} \right)+F^{rs}\nabla_{\overline{\ell}}\nabla_{k}\left(g_{r\overline{j}}\alpha^{\overline{j}s} \right)-\psi_{k\overline{\ell}}\\
=&F^{rs,ab}\left(\alpha^{\overline{j}s} \nabla_kg_{r\overline{j}}
\right)\left(\alpha^{\overline{c}b}\nabla_{\overline{\ell}}g_{a\overline{c}}\right)
 +F^{rs}\alpha^{\overline{j}s}\nabla_{\overline{\ell}}\nabla_{k}g_{r\overline{j}}-\psi_{k\overline{\ell}}\nonumber\\
=&F^{rs,ab}\left(\nabla_kg_{r\overline{s}}\right)\left(\nabla_{\overline{\ell}}g_{a\overline{b}}\right)
 +F^{rr}\nabla_{\overline{\ell}}\nabla_{k}g_{r\overline{r}}-\psi_{k\overline{\ell}}.\nonumber
\end{align}
Combining \eqref{eigenvalue1st}, \eqref{dotuk} and \eqref{dotukl} gives
\begin{align}\label{dotlambda}
\partial_t\tilde{\lambda}_{1}
=&\tilde\lambda_{1}^{pq}\partial_t\left(g_{p\overline{j}}\alpha^{\overline{j}q}\right)\\
=&\tilde\lambda_{1}^{pq} \left(\dot{u}_{p\overline{j}}+W_{p\overline{j}}^{r}\dot{u}_r+\overline{W_{j\overline{p}}^{r}}\dot{u}_{\overline{r}}\right)\alpha^{\overline{j}q}\nonumber\\
=&W_{1\overline{1}}^{r}\dot{u}_r+\overline{W_{1\overline{1}}^{r}}\dot{u}_{\overline{r}}+\dot{u}_{1\overline{1}}\nonumber\\
=&W_{1\overline{1}}^{r}\left(F^{qq}\nabla_{r}g_{q\overline{q}} -\psi_r\right)
+\overline{W_{1\overline{1}}^{r}}\left(F^{qq}\nabla_{\overline{r}}g_{q\overline{q}}-\psi_{\overline{r}}\right)\nonumber\\
&+F^{rs,ab}\left(\nabla_1g_{r\overline{s}}\right)\left(\nabla_{\overline{1}}g_{a\overline{b}}\right)
+F^{qq}\nabla_{\overline{1}}\nabla_{1}g_{q\overline{q}}-\psi_{1\overline{1}}\nonumber
\end{align}
Thanks to \eqref{dotuk} and  \eqref{dotlambda}, it follows
\begin{align}\label{dotH}
\partial_t\tilde H
=&\frac{\partial_t\tilde{\lambda}_{1}}{\lambda_1} +\phi'( u_r\dot{u}_{\bar r } +
   \dot{u}_{r}u_{\bar r} )  +\varphi' \dot{u},\\
=&\frac{1}{\lambda_1}\left[W_{1\overline{1}}^{r}\left(F^{qq}\nabla_{r}g_{q\overline{q}} -\psi_r\right)
+\overline{W_{1\overline{1}}^{r}}\left(F^{qq}\nabla_{\overline{r}}g_{q\overline{q}}-\psi_{\overline{r}}\right)
+F^{qq}\nabla_{\overline{1}}\nabla_{1}g_{q\overline{q}}-\psi_{1\overline{1}}\right]\nonumber\\
&+\frac{1}{\lambda_1} F^{rs,ab}\left(\nabla_1g_{r\overline{s}}\right)\left(\nabla_{\overline{1}}g_{a\overline{b}}\right)\nonumber\\
&+\phi' u_{\bar r }\left(F^{qq}\nabla_{r}g_{q\overline{q}}-\psi_r\right)
 +\phi' u_r\left(F^{qq}\nabla_{\overline{r}}g_{q\overline{q}}  -\psi_{\overline{r}}\right) +\varphi' \dot{u}.\nonumber
\end{align}
Combining \eqref{hqq}, \eqref{lambda1qq} and \eqref{dotH}, we get
\begin{align}\label{fqqhqqdoth}
0\geq&F^{qq}\tilde H_{q\overline{q}}-\partial_t\tilde H\\
=&-\frac{F^{qq}|\tilde{\lambda}_{1,q}|^2}{\lambda_1^2}+
\frac{1}{\lambda_1}F^{qq}\left(\nabla_{\overline{q}}\nabla_qg_{1\overline{1}}-\nabla_{\overline{1}}\nabla_1g_{q\overline{q}}\right) -\frac{1}{\lambda_1}  F^{rs,ab}\left(\nabla_1g_{r\overline{s}}\right)\left(\nabla_{\overline{1}}g_{a\overline{b}}\right)\nonumber\\
&-\frac{1}{\lambda_1}\left(W_{1\overline{1}}^{r}\left(F^{qq}\nabla_{r}g_{q\overline{q}} -\psi_r\right)
+\overline{W_{1\overline{1}}^{r}}\left(F^{qq}\nabla_{\overline{r}}g_{q\overline{q}}-\psi_{\overline{r}}\right)\right)
+\frac{1}{\lambda_1}\psi_{1\overline{1}}\nonumber\\
 &+ \frac{F^{qq}}{\lambda_1}\left(\sum\limits_{r> 1}\frac{|\nabla_{q}g_{r\overline{1}} |^2+|\nabla_{q}g_{1\overline{r}}|^2}{\lambda_1-\tilde\lambda_{r}}
 -\sum\limits_{r> 1}\frac{\left(\nabla_{p}B_{r}{}^1\right)\left(\nabla_{\overline{q}}g_{1\overline{r}}\right)
 +\left(\nabla_{p}B_{1}{}^r\right)\left(\nabla_{\overline{q}}g_{r\overline{1}}\right)}{\lambda_1-\tilde\lambda_{r}}\right)\nonumber\\
&+ \phi'F^{qq}\Big(u_r\nabla_{\overline{q}}u_{\overline{r}q} + u_{\overline{r}} \nabla_{\overline{q}}\nabla_{q}u_{r}  +
|\nabla_{q}u_{r}|^2 + |u_{\overline{r}q}|^2 \Big) \nonumber \\
& + \phi''F^{qq}|V_q|^2 + \varphi''F^{qq}|u_q|^2 + \varphi' F^{qq}u_{q\overline{q}}  \nonumber \\
&-\phi' u_{\bar r }\left(F^{qq}\nabla_{r}g_{q\overline{q}}-\psi_r\right)
 -\phi' u_r\left(F^{qq}\nabla_{\overline{r}}g_{q\overline{q}}  -\psi_{\overline{r}}\right) -\varphi' \dot{u}.\nonumber
\end{align}
Using \eqref{ricciidentity} implies
\begin{align}\label{qqg11-11gqq}
&\nabla_{\overline{q}}\nabla_qg_{1\overline{1}}-\nabla_{\overline{1}}\nabla_1g_{q\overline{q}}\\
=&\nabla_{\overline{q}}\nabla_q\chi_{1\overline{1}}-\nabla_{\overline{1}}\nabla_1\chi_{q\overline{q}}
 +\nabla_{\overline{q}}\nabla_qu_{1\overline{1}}-\nabla_{\overline{1}}\nabla_1u_{q\overline{q}}
 +\nabla_{\overline{q}}\nabla_qW(u)_{1\overline{1}}-\nabla_{\overline{1}}\nabla_1W(u)_{q\overline{q}}\nonumber\\
=&\nabla_{\overline{q}}\nabla_q\chi_{1\overline{1}}-\nabla_{\overline{1}}\nabla_1\chi_{q\overline{q}}
 -2\mathrm{Re}\left(\overline{T_{q1}^r}\nabla_{1}u_{q\overline{r}}\right)+R_{q\overline{q}1}{}^pu_{p\overline{1}}
 -R_{1\overline{1}q}{}^pu_{p\overline{q}}
 -T_{1q}^p\overline{T_{q1}^r}u_{p\overline{r}} \nonumber\\
 &+\nabla_{\overline{q}}\nabla_qW(u)_{1\overline{1}}-\nabla_{\overline{1}}\nabla_1W(u)_{q\overline{q}}.\nonumber
\end{align}
Since $W_{1\overline{1}}=W_{1\overline{1}}^{r}u_r+\overline{W_{1\overline{1}}^{r}}u_{\overline{r}}$, using the Ricci identity \eqref{ricciidentity}, we can get
\begin{align}
\nabla_{\overline{q}}\nabla_{q}W_{1\overline{1}}
=&u_r\nabla_{\overline{q}}\nabla_{q}W_{1\overline{1}}^r+u_{r\overline{q}}\nabla_{q}W_{1\overline{1}}^r
+\left(\nabla_{\overline{q}} W_{1\overline{1}}^r\right)\left(\nabla_{q}u_r\right)+W_{1\overline{1}}^r\nabla_{\overline{q}}\nabla_{q}u_r\nonumber\\
&+u_{\overline{r}}\nabla_{\overline{q}}\nabla_{q}\overline{W_{1\overline{1}}^r}
+\left(\nabla_{\overline{q}}u_{\overline{r}}\right)\left(\nabla_{q}\overline{W_{1\overline{1}}^r}\right)
+u_{\overline{r}q} \nabla_{\overline{q}} \overline{W_{1\overline{1}}^r}
+\overline{W_{1\overline{1}}^r}\nabla_{\overline{q}} u_{\overline{r}q}\nonumber\\
=&u_r\nabla_{\overline{q}}\nabla_{q}W_{1\overline{1}}^r+u_{r\overline{q}}\nabla_{q}W_{1\overline{1}}^r
+\left(\nabla_{\overline{q}} W_{1\overline{1}}^r\right)\left(\nabla_{q}u_r\right)\nonumber\\
&+W_{1\overline{1}}^r\nabla_ru_{q\overline{q}}-W_{1\overline{1}}^rT_{qr}^pu_{p\overline{q}}+W_{1\overline{1}}^rR_{q\overline{q}r}{}^pu_p\nonumber\\
&+u_{\overline{r}}\nabla_{\overline{q}}\nabla_{q}\overline{W_{1\overline{1}}^r}
+\left(\nabla_{\overline{q}}u_{\overline{r}}\right)\left(\nabla_{q}\overline{W_{1\overline{1}}^r}\right)
+u_{\overline{r}q} \nabla_{\overline{q}} \overline{W_{1\overline{1}}^r}\nonumber\\
&+\overline{W_{1\overline{1}}^r}\nabla_{\overline{r}}u_{q\overline{q}}-\overline{W_{1\overline{1}}^r}\overline{T_{qr}^p}u_{q\overline{p}} \nonumber\\
=&u_r\nabla_{\overline{q}}\nabla_{q}W_{1\overline{1}}^r+u_{r\overline{q}}\nabla_{q}W_{1\overline{1}}^r
+\left(\nabla_{\overline{q}} W_{1\overline{1}}^r\right)\left(\nabla_{q}u_r\right)\nonumber\\
&+W_{1\overline{1}}^r\left(\nabla_rg_{q\overline{q}}-\nabla_r\chi_{q\overline{q}}-\nabla_rW(u)_{q\overline{q}}\right)
-W_{1\overline{1}}^rT_{qr}^pu_{p\overline{q}}+W_{1\overline{1}}^rR_{q\overline{q}r}{}^pu_p\nonumber\\
&+u_{\overline{r}}\nabla_{\overline{q}}\nabla_{q}\overline{W_{1\overline{1}}^r}
+\left(\nabla_{\overline{q}}u_{\overline{r}}\right)\left(\nabla_{q}\overline{W_{1\overline{1}}^r}\right)
+u_{\overline{r}q} \nabla_{\overline{q}} \overline{W_{1\overline{1}}^r}\nonumber\\
&+\overline{W_{1\overline{1}}^r}\left(\nabla_{\overline{r}}g_{q\overline{q}}-\nabla_{\overline{r}}\chi_{q\overline{q}}-
\nabla_{\overline{r}}W(u)_{q\overline{q}}\right)-\overline{W_{1\overline{1}}^r}\overline{T_{qr}^p}u_{q\overline{p}}, \nonumber
\end{align}
which implies
\begin{align}\label{fqqwqq11}
 F^{qq}\nabla_{\overline{q}}\nabla_{q}W_{1\overline{1}}-F^{qq}W_{1\overline{1}}^r \nabla_rg_{q\overline{q}}-F^{qq}\overline{W_{1\overline{1}}^r} \nabla_{\overline{r}}g_{q\overline{q}}
\geq -C\left(\sum\limits_{r} F^{qq}|\nabla_{q}u_r|+\lambda_1\mathcal{F}\right),
\end{align}
where we use the fact that $\lambda_1\gg K>1$ and hence that $|u_{p\overline{q}}|$ can be controlled by $\lambda_1$.

Since
$$
W(u)_{i\overline{j}}=\left(\tr_{\alpha}Z(u)\right)\alpha_{i\overline{j}}-(n-1)Z(u)_{i\overline{j}},
$$
\eqref{daof} implies (as in \cite{stw1503})
\begin{align}\label{fkkaawkk}
F^{qq}\nabla_{\overline{1}}\nabla_{1}W(u)_{q\overline{q}}
\leq C\left(\sum\limits_{r}F^{qq}|\nabla_qu_r|+ |\nabla_{q}g_{1\overline{1}}|+\lambda_1\mathcal{F}\right).
\end{align}
From Young's inequality, it follows that
\begin{align}
&\sum\limits_{r> 1}\frac{|\nabla_{q}g_{r\overline{1}} |^2+|\nabla_{q}g_{1\overline{r}}|^2}{\lambda_1-\tilde\lambda_{r}}
 -\sum\limits_{r> 1}\frac{\left(\nabla_{p}B_{r}{}^1\right)\left(\nabla_{\overline{q}}g_{1\overline{r}}\right)
 +\left(\nabla_{p}B_{1}{}^r\right)\left(\nabla_{\overline{q}}g_{r\overline{1}}\right)}{\lambda_1-\tilde\lambda_{r}}\nonumber\\
\geq&\frac{1}{2}\sum\limits_{r> 1}\frac{|\nabla_{q}g_{r\overline{1}} |^2+|\nabla_{q}g_{1\overline{r}}|^2}{\lambda_1-\tilde\lambda_{r}}-C,\nonumber
\end{align}
where we also use \eqref{6.18}.
Since
$$
(n-1)\lambda_1+\tilde\lambda_r\geq \sum\limits_{i=1}^n\tilde\lambda_i>-1,
$$
we get $(\lambda_1-\tilde\lambda_r)^{-1}\geq (n\lambda_1+1)^{-1}$ for $r>1$. Then Young's inequality gives
\begin{align}\label{lambda1rsab}
&\sum\limits_{r> 1}\frac{|\nabla_{q}g_{r\overline{1}} |^2+|\nabla_{q}g_{1\overline{r}}|^2}{\lambda_1-\tilde\lambda_{r}}
 -\sum\limits_{r> 1}\frac{\left(\nabla_{p}B_{r}{}^1\right)\left(\nabla_{\overline{q}}g_{1\overline{r}}\right)
 +\left(\nabla_{p}B_{1}{}^r\right)\left(\nabla_{\overline{q}}g_{r\overline{1}}\right)}{\lambda_1-\tilde\lambda_{r}}\\
 \geq& \frac{1}{2(n\lambda_1+1)}\sum\limits_{r> 1}\left(|\nabla_{q}g_{r\overline{1}} |^2+|\nabla_{q}g_{1\overline{r}}|^2\right)-C\nonumber\\
 \geq& \frac{1}{4n\lambda_1}\sum\limits_{r> 1}\left(|\nabla_{q}g_{r\overline{1}} |^2+|\nabla_{q}g_{1\overline{r}}|^2\right)-C,\nonumber
\end{align}
where for the second  inequality we use the fact that $\lambda_1>1$.

Using the Ricci identity, we can obtain
\begin{align*}
\overline{T_{q1}^r}\nabla_{1}u_{q\overline{r}}
=&\overline{T_{q1}^r}\nabla_{q}u_{1\overline{r}}-\overline{T_{q1}^r}T_{1q}^pu_{p\overline{r}}\\
=&\overline{T_{q1}^r}\nabla_{q}g_{1\overline{r}}-\overline{T_{q1}^r}\nabla_{q}\chi_{1\overline{r}}
-\overline{T_{q1}^r}\nabla_{q}W(u)_{1\overline{r}}-\overline{T_{q1}^r}T_{1q}^pu_{p\overline{r}},\nonumber
\end{align*}
which implies
\begin{align}\label{reyong}
-\frac{2}{\lambda_1}\mathrm{Re}\left(F^{qq}\overline{T_{q1}^r}\nabla_{1}u_{q\overline{r}}\right)
\geq &-\frac{2 F^{qq}}{\lambda_1}\mathrm{Re}\left(\overline{T_{q1}^r}\nabla_{q}g_{1\overline{r}} \right)
 -C\left(\frac{1}{\lambda_1}\sum\limits_{r}F^{qq}|\nabla_{q}u_r|+ \mathcal{F}\right),
\end{align}
where we use the fact that $|u_{i\overline{j}}|$ can be controlled by $\lambda_1$.
Then we have
\begin{align}\label{resolve}
\frac{\eqref{lambda1rsab}}{\lambda_1}+\eqref{reyong}\geq -\frac{2 F^{qq}}{\lambda_1}\mathrm{Re}\left(\overline{T_{q1}^1}\nabla_{q}g_{1\overline{1}} \right)
 -C\left(\frac{1}{\lambda_1}\sum\limits_{r}F^{qq}|\nabla_{q}u_r|+ \mathcal{F}\right),
\end{align}
where we use that $\mathcal{F}\geq \kappa$ and hence absorb the constant $C$ into $C\mathcal{F}$.

Substituting \eqref{qqg11-11gqq}, \eqref{fkkaawkk}, \eqref{fqqwqq11} and \eqref{resolve} into \eqref{fqqhqqdoth}, it follows that
\begin{align}\label{fqqhqqdoth2}
0\geq&-\frac{F^{qq}|\tilde{\lambda}_{1,q}|^2}{\lambda_1^2}
-\frac{1}{\lambda_1}  F^{rs,ab}\left(\nabla_1g_{r\overline{s}}\right)\left(\nabla_{\overline{1}}g_{a\overline{b}}\right)\\
&  -C\mathcal{F}-\frac{C}{\lambda_1}\left(\sum\limits_r F^{qq}|\nabla_qu_r|+F^{qq}|\nabla_qg_{1\overline{1}}|\right)\nonumber\\
&+ \phi'F^{qq}\Big(u_r\nabla_{\overline{q}}u_{\overline{r}q} + u_{\overline{r}} \nabla_{\overline{q}}\nabla_{q}u_{r}  +
|\nabla_{q}u_{r}|^2 + |u_{\overline{r}q}|^2 \Big) \nonumber \\
& + \phi''F^{qq}|V_q|^2 + \varphi''F^{qq}|u_q|^2 + \varphi' F^{qq}u_{q\overline{q}}  \nonumber \\
&-\phi' u_{\bar r }\left(F^{qq}\nabla_{r}g_{q\overline{q}}-\psi_r\right)
 -\phi' u_r\left(F^{qq}\nabla_{\overline{r}}g_{q\overline{q}}  -\psi_{\overline{r}}\right) -\varphi' \dot{u}.\nonumber
\end{align}
Using the Ricci identity, we can get
\begin{align*}
\phi'F^{qq}u_r\nabla_{\overline{q}}u_{\overline{r}q}
=&\phi'F^{qq}u_r\left(\nabla_{\overline{r}}u_{q\overline{q}}-\overline{T_{qr}^p}u_{\overline{p}q}\right)\\
=&\phi'F^{qq}u_r\left(\nabla_{\overline{r}}g_{q\overline{q}}-\nabla_{\overline{r}}\chi_{q\overline{q}}
-\nabla_{\overline{r}}W(u)_{q\overline{q}}-\overline{T_{qr}^p}u_{q\overline{p}}\right),
\end{align*}
which implies
\begin{align}\label{fenleiqian}
\phi'F^{qq}u_r\left(\nabla_{\overline{q}}u_{\overline{r}q}-\nabla_{\overline{r}}g_{q\overline{q}}\right)
\geq -\frac{C}{K^{1/2}}\left(\sum\limits_{r}|u_{r\overline{q}} |F^{qq}+\sum\limits_{r}|\nabla_qu_r|F^{qq}+\mathcal{F}\right)
\end{align}
where we use  the fact that $\phi'\leq 1/(2K)$. Here we also use that $|u_r|$ can be controlled by $\lambda_1$ and hence can also be controlled by $|u_{p\overline{q}}|$.

Similarly, we can also deduce
\begin{align}
\phi'F^{qq}u_{\overline{r}} \nabla_{\overline{q}}\nabla_{q}u_{r}
=&\phi'F^{qq}u_{\overline{r}}\nabla_{\overline{q}}\left(\nabla_{r}u_{q}-T_{qr}^pu_p\right)\nonumber\\
=&\phi'F^{qq}u_{\overline{r}}\left(\nabla_{r}u_{q\overline{q}}-R_{r\overline{q}q}{}^pu_p
-u_p\nabla_{\overline{q}}T_{qr}^p-T_{qr}^pu_{p\overline{q}}\right)\nonumber\\
=&\phi'F^{qq}u_{\overline{r}}\left(\nabla_{r}g_{q\overline{q}}-\nabla_{r}\chi_{q\overline{q}}-\nabla_{r}W(u)_{q\overline{q}}-R_{r\overline{q}q}{}^pu_p
-u_p\nabla_{\overline{q}}T_{qr}^p-T_{qr}^pu_{p\overline{q}}\right),\nonumber
\end{align}
which implies
\begin{align}\label{fenleiqian1}
\phi'F^{qq}u_{\overline{r}}\left(\nabla_{\overline{q}}\nabla_{q}u_{r}-\nabla_{r}g_{q\overline{q}}\right)
\geq -\frac{C}{K^{1/2}}\left(\sum\limits_{r}|u_{r\overline{q}} |F^{qq}+\sum\limits_{r}|\nabla_qu_r|F^{qq}+\mathcal{F}\right).
\end{align}
From \eqref{fqqhqqdoth2}, \eqref{fenleiqian}, \eqref{fenleiqian1} and the fact that $\phi'\geq 1/(4K)$, we can get
\begin{align}\label{fqqhqqdoth3}
0\geq&-\frac{F^{qq}|\tilde{\lambda}_{1,q}|^2}{\lambda_1^2}
-\frac{1}{\lambda_1}  F^{rs,ab}\left(\nabla_1g_{r\overline{s}}\right)\left(\nabla_{\overline{1}}g_{a\overline{b}}\right)\\
&  -C\left(\mathcal{F} +\lambda_1^{-1}F^{qq}|\nabla_qg_{1\overline{1}}|\right)+\sum\limits_{r}\frac{F^{qq}}{6K}\Big(|\nabla_{q}u_{r}|^2 + |u_{\overline{r}q}|^2 \Big)\nonumber\\
& + \phi''F^{qq}|V_q|^2 +\varphi''F^{qq}|u_q|^2 + \varphi' F^{qq}u_{q\overline{q}} -\varphi' \dot{u},\nonumber
\end{align}
where we use the fact that $K>1$ and hence $K^{-1/2}<1$ and absorb the constant into $C\mathcal{F}$.

Next, we can use the ideas of \cite{stw1503} in the elliptic setting to deal with two cases separately (cf. \cite{houmawu,sz}).

\textbf{Case 1.} Suppose that $\delta\lambda_1\geq -\lambda_n$, where the constant $\delta$ will be determined later. First we define the set
$$
I:=\Big\{i:\;F^{ii}>\delta^{-1}F^{11}\Big\}.
$$
First we remark that for $i\in I$, we have $F^{kk}>F^{11}$ and hence it follows that  $i>1$ and $\lambda_1>\lambda_i$. Using the same discussion as in \cite{stw1503}, from \eqref{fqqhqqdoth3}, we can deduce
\begin{align}\label{3.35}
0\geq&-(1-2\delta)\sum\limits_{q\in I}\frac{F^{qq}\left(|\tilde{\lambda}_{1,q}|^2-|\nabla_1g_{q\overline{1}}|^2\right)}{\lambda_1^2} \\
&  -C\left(\mathcal{F} +\lambda_1^{-1}F^{kk}|\nabla_kg_{1\overline{1}}|\right)+\frac{F^{qq}}{6K}\sum\limits_{r}\Big(|\nabla_{q}u_{r}|^2 + |u_{\overline{r}q}|^2 \Big)\nonumber\\
&-2\varphi'^2\delta^{-1}F^{11}K+\frac{1}{2}\varphi''F^{qq}|u_q|^2 + \varphi' F^{qq}u_{q\overline{q}} -\varphi' \dot{u}.\nonumber
\end{align}
Here we choose  the positive constant $\delta$ small enough so that
\begin{align}\label{deltadefn}
4\delta \varphi'^2\leq \frac{1}{2}\varphi'',
\end{align}
which can be easily obtained from the definition of $\varphi$.

Then we need the first claim whose proof is the same as the one in \cite{stw1503}.

\textbf{Claim.} If $\lambda_1/K$ is sufficiently large compared to $\varphi'$ (i.e., the choices of $D_1$ and $D_2$ which will be determined later), then for any $\varepsilon>0$, there exists a constant $C_{\varepsilon}$ with
\begin{align}\label{claim}
\sum\limits_{q\in I}\frac{F^{qq}|\nabla_{1}g_{q\overline{1}}|^2}{\lambda_1^2}
\geq&\sum\limits_{q\in I}\frac{F^{qq}|\tilde\lambda_{1,q}|^2}{\lambda_1^2}-
 \sum\limits_{r}\frac{F^{qq}}{12K}\left(|u_{\overline{r}q}|^2+|\nabla_{q}u_{r}|^2\right)\\
&+C_{\varepsilon}\varphi'F^{qq}|u_q|^2+\varepsilon C\varphi'\mathcal{F}-C\mathcal{F}.\nonumber
\end{align}
Substituting \eqref{claim} into \eqref{3.35} gives
\begin{align}\label{3.52}
0\geq& \sum\limits_{r}\frac{F^{qq}}{12K}\Big(|\nabla_{q}u_{r}|^2 + |u_{\overline{r}q}|^2 \Big)
+C_{\varepsilon}\varphi'F^{qq}|u_q|^2+\varepsilon C\varphi'\mathcal{F}\\
&  -C\left(\mathcal{F} +\lambda_1^{-1}F^{qq}|\nabla_qg_{1\overline{1}}|\right)-2\varphi'^2\delta^{-1}F^{11}K\nonumber\\
&+\frac{1}{2}\varphi''F^{qq}|u_q|^2 + \varphi' F^{qq}u_{q\overline{q}} -\varphi' \dot{u}.\nonumber
\end{align}
Using \eqref{hq} again implies
\begin{align}\label{lambda1qnqw11use2}
\frac{|\tilde\lambda_{1,q}|}{\lambda_1}
=&\left|\phi'( u_ru_{\overline{r}q} +u_{\overline{r}}\nabla_qu_{r})+\varphi' u_q\right|\\
\leq&\frac{1}{2K^{1/2}}\left(\sum\limits_{p=1}^n\sum\limits_{r=2}^n|\nabla_{r}u_p|\right) +\frac{1}{2K^{1/2}}\sum\limits_{r}|u_{\overline{r}q}|-C_{\varepsilon}\varphi' |u_q|^2-\varepsilon\varphi',\nonumber
\end{align}
where we use Young's inequality and the facts that $\phi'\leq(2K)^{-1}$ and that $\varphi'<0$. This inequality and
\eqref{lambda1p}  imply
\begin{align*}
 F^{qq}\lambda_{1}^{-1}|\nabla_qg_{1\overline{1}}|\leq   \frac{1}{2K^{1/2}}\sum\limits_{r}\left(|u_{\overline{r}q}|+|\nabla_qu_r|\right)-C_{\varepsilon}\varphi' F^{qq}|u_q|^2-\varepsilon\varphi'\mathcal{F}.
\end{align*}
Substituting this into \eqref{3.52} implies
\begin{align}\label{0geq1}
0\geq& \sum\limits_{r}\frac{F^{qq}}{20K}\Big(|\nabla_{q}u_{r}|^2 + |u_{\overline{r}q}|^2 \Big)
+C_{\varepsilon}\varphi'F^{qq}|u_q|^2+\varepsilon C\varphi'\mathcal{F}\\
&  -C \mathcal{F}  -2\varphi'^2\delta^{-1}F^{11}K
 +\frac{1}{2}\varphi''F^{qq}|u_q|^2 + \varphi' F^{qq}u_{q\overline{q}} -\varphi' \dot{u}.\nonumber
\end{align}
Since $W(u)_{i\overline{j}}=\left(\tr_{\alpha}Z(u)\right)\alpha_{i\overline{j}}-(n-1)Z(u)_{i\overline{j}}$, \eqref{zipq} and \eqref{daof} gives
\begin{align}
\sum\limits_{q}F^{qq}W(u)_{q\overline{q}}
=&\sum\limits_{q}\tilde F^{qq}Z(u)_{q\overline{q}}\nonumber\\
=&\sum\limits_{q=1}^n\tilde F^{qq}\left(\sum\limits_{r\not=q}Z_{q\overline{q}}^ru_r\right)
+\sum\limits_{q=1}^n\tilde F^{qq}\left(\sum\limits_{r\not=q}\overline{Z_{q\overline{q}}^r}u_{\overline{r}}\right)\nonumber\\
=&\sum\limits_{r=1}^n\left(\sum\limits_{q\not=r}\tilde F^{qq}Z_{q\overline{q}}^r\right)u_r
+\sum\limits_{r=1}^n\left(\sum\limits_{q\not=r}\tilde F^{qq}\overline{Z_{q\overline{q}}^r}\right)u_{\overline{r}}.\nonumber
\end{align}
This together with \eqref{tildef1fkf1} and \eqref{tildefk} implies
\begin{align*}
\left|\sum\limits_{q}F^{qq}W(u)_{q\overline{q}}\right|\leq C\sum\limits_{q}F^{qq}|u_q|\leq C_{\varepsilon}F^{qq}|u_q|^2+\varepsilon\mathcal{F},
\end{align*}
where in the last step we use   Young's inequality. Then it follows that
\begin{align}\label{varphifqquqq}
\varphi' F^{qq}u_{q\overline{q}}
=&\varphi'F^{qq}\left(g_{q\overline{q}}-\chi_{q\overline{q}}-W(u)_{q\overline{q}}\right)\\
\geq& \varphi'F^{qq}\left(g_{q\overline{q}}-\chi_{q\overline{q}}\right)+C_{\varepsilon}\varphi'F^{qq}|u_q|^2+\varepsilon\varphi'\mathcal{F},\nonumber
\end{align}
where we use the fact that $\varphi'<0$.
Substituting \eqref{varphifqquqq} into \eqref{0geq1} gives
\begin{align}\label{0geq2}
0\geq& F^{11}\left(\frac{\lambda_1^2}{40K} -2\varphi'^2\delta^{-1}K \right)
+\left(C_{\varepsilon}\varphi'+\frac{1}{2}\varphi''\right)F^{qq}|u_q|^2+\varepsilon C\varphi'\mathcal{F}\\
&  -C \mathcal{F}  - \varphi' F^{qq}\left(\chi_{q\overline{q}}-g_{q\overline{q}}\right) -\varphi' \dot{u},\nonumber
\end{align}
where we use the fact that $|u_{1\overline{1}}|\geq \frac{1}{2}\lambda_{1}^2-CK$.

Now we need to use Proposition \ref{prop}.
\begin{enumerate}
\item[(a)] Suppose that we have $F^{qq}\left(\chi_{q\overline{q}}-g_{q\overline{q}}\right)-|\dot u|>\kappa\mathcal{F}$. Then \eqref{0geq2} becomes
\begin{align*}
0\geq& F^{11}\left(\frac{\lambda_1^2}{40K} -2\varphi'^2\delta^{-1}K \right)
+\left(C_{\varepsilon}\varphi'+\frac{1}{2}\varphi''\right)F^{qq}|u_q|^2+\varepsilon C\varphi'\mathcal{F}\\
&  -C \mathcal{F}  - \varphi' \kappa \mathcal{F}.\nonumber
\end{align*}
Choose $\varepsilon >0$ sufficiently small such that $\varepsilon C<\frac{\kappa}{2}$. Next we choose $D_2$ in the definition $\varphi(t)=D_1e^{-D_2t}$ to be large enough with
$$
\frac{1}{2}\varphi''>C_{\varepsilon}|\varphi'|.
$$
Then we arrive at
\begin{align*}
0\geq& F^{11}\left(\frac{\lambda_1^2}{40K} -2\varphi'^2\delta^{-1}K \right)
  -C \mathcal{F}  - \frac{1}{2}\varphi' \kappa \mathcal{F}.
\end{align*}
Now we choose $D_1$ large enough such that $-\frac{1}{2}\kappa \varphi'>C$, then we can get
\begin{align*}
0\geq&  \frac{\lambda_1^2}{40K} -2\varphi'^2\delta^{-1}K,
\end{align*}
which implies the bound of $\lambda_1/K$, as required.
\item[(b)] Suppose that we have $F^{11}>\kappa \mathcal{F}$. With the constants  $D_1$ and $D_2$  determined above, $|\dot u|$ and $\varphi'$ is uniformly bounded and hence can be absorbed in $C\mathcal{F}$. Then using the same arguments in \cite{stw1503}, we can get the uniform upper bound of $\lambda_1/K$ and omit the details here.
\end{enumerate}
\textbf{Case 2.} Suppose that $\delta \lambda_1\leq -\lambda_n$, and all the constants $D_1,\,D_2$ and $\delta$ are fixed as in the previous case. Since in this case we can absorb the uniformly bounded term $-\varphi'\dot u$ into $C\mathcal{F}$, using the same discussion as in \cite{stw1503}, we can get the uniform upper bound of $\lambda_1/K$. Here we omit the details.

\end{proof}
\section{First order estimate}\label{secfirst}
Given the form of the second estimate, we need the first order estimate for $u$.
\begin{thm}\label{1stestimate}
Let $u$ be the solution to \eqref{paragau} on $M\times[0,\,T)$. Then there exists a uniform $C>0$ such that
\begin{equation}
\label{formular1stestimate}\sup\limits_{M\times[0,\,T)}|\nabla u|_{\alpha}^2\leq C.
\end{equation}
\end{thm}
First we recall some notations from \cite{tw1305}. Let $\beta$ be the Euclidean K\"{a}hler metric on $\mathbb{C}^n,\;n\geq 2$ and $\Delta$ be the Laplace operator with respect to $\beta$. Let $\Omega\subset \mathbb{C}^n$ be a domain. Suppose that
$$
u\longrightarrow\mathbb{R}\cup\{-\infty\}
$$
is an upper semicontinuous function which is in $L^{1}_{loc}(\Omega)$. If
$$
P(u):=\frac{1}{n-1}\left((\Delta u)\beta-\ddbar u\right)\geq 0
$$
as a real $(1,1)$ current, then we will say that $u$ is $(n-1)$-Psh.
\begin{defn}
Let $u$ be a continuous $(n-1)$-Psh function. Then we say that $u$ is \emph{maximal} if for any relatively compact set $\Omega'\subset\subset \Omega$ and any continuous $(n-1)$-Psh function $v$ on a domain $\Omega''$ with $\Omega'\subset\subset \Omega''\subset\subset\Omega$ and $v\leq u$ on $\partial \Omega'$, then we must have $v\leq u$ on $\Omega'$.
\end{defn}
We need the following Liouville type theorem from \cite{tw1305}.
\begin{thm}[Tosatti and Weinkove \cite{tw1305}; 2013]\label{ltthm}
Let $u:\;\mathbb{C}^n\longrightarrow \mathbb{R}$ be an $(n-1)$-Psh function which is Lipschitz continuous, maximal, and satisfies
$$
\sup\limits_{\mathbb{C}^n}(|u|+|\nabla u|)<\infty.
$$
Then $u$ is a constant.
\end{thm}
Using the idea of Dinew and Ko{\l}odziej \cite{DK} and the Liouville type Theorem \ref{ltthm}, the argument is identical to \cite{gill1410} which can be obtained by modifying the argument of  \cite{tw1305} in the elliptic setting. Hence, we can be brief and just point out the main differences.
\begin{proof}[Proof of Theorem \ref{1stestimate}]
Suppose that \eqref{formular1stestimate} does not hold. Then there exists a sequence $(x_j,t_j)\in M\times[0,T)$ with $t_j\longrightarrow T$ such that
\begin{align*}
\lim_{j\rightarrow\infty}|\nabla u(x_j,t_j)|_{\alpha}=+\infty.
\end{align*}
Without loss of generality, we assume that
\begin{align*}
C_j:=|\nabla u(x_j,t_j)|_{\alpha}=\sup\limits_{(x,t)\in M\times[0,t_j]}|\nabla u(x,t)|_{\alpha}\longrightarrow +\infty,\quad\text{as}\quad j\longrightarrow \infty
\end{align*}
and
$
\lim\limits_{j\rightarrow \infty}x_j=x.
$
Fix holomorphic coordinates $z=(z^1,\cdots,z^n)$ centered at $x$, i.e., $z(x)=0$, and $\alpha(x)=\beta$ identifying with $B_2(0)$.
Also assume that $j$ large enough so that $x_j\in B_1(0)$. Define
\begin{align*}
u_j:\;&M\longrightarrow \mathbb{C},\quad\text{given by}\quad u_j(y):=u(y,t_j),\\
\Phi_j:\;&\mathbb{C}^n\longrightarrow \mathbb{C}^n,\quad\text{given by}\quad \Phi_j(z):=C_{j}^{-1}z+x_j=:w,\\
\hat u_j:\;&\mathbb{C}^n\supset B_{C_j}(0)\longrightarrow \mathbb{C},\quad\text{given by}\quad\hat u_j(z):=u_j\circ \Phi_j(z)=u_j\left(C_{j}^{-1}z+x_j\right).
\end{align*}
Note that for convenience, we confuse $z(x)$ and $z$ by the local coordinates. The function $\hat u_j$ satisfies
$$
\sup_{B_{C_j}(0)}|\hat u_j|\leq C\quad \text{and}\quad \sup_{B_{C_j}(0)}|\nabla \hat u_j|_{\beta}\leq CC_j^{-1}|\nabla u_j|_{\alpha}(x_j+C_j^{-1}z)\leq C,
$$
where we use the fact that the Euclidean metric $\beta$ is equivalent to the Hermitian metric $\alpha$ in $B_2(0)$. In particular,
we have
\begin{align}\label{7.2}
\sup_{B_{C_j}(0)}|\nabla \hat u_j|_{\beta}(0)\geq C^{-1}C_{j}^{-1}|\nabla u_j|_{\alpha}(x_j)=C^{-1}>0.
\end{align}
Thanks to Theorem \ref{2ndestimate}, we know that
$$
|\ddbar \hat u_j|_{\beta}\leq CC_{j}^{-2}|\ddbar   u_j|_{\alpha}\leq C'.
$$
Then the elliptic estimate for $\Delta$ and Sobolev embedding gives that for each compact $K\subset\mathbb{C}^n$, each $\gamma\in(0,1)$ and $p>1$, there exists a constant $C$ with
$$
\|\hat u_j\|_{C^{1,\gamma}(K)}+\|\hat u_j\|_{W^{2,p}(K)}\leq C.
$$
Therefore, there exists a subsequence of $\hat u_j$ converges strongly in $C^{1,\gamma}(\mathbb{C}^n)$ and weakly in $W^{2,p}_{\mathrm{loc}}(\mathbb{C}^n)$ to a function $u\in W^{2,p}_{\mathrm{loc}}(\mathbb{C}^n)$ with $\sup_{\mathbb{C}^n}(|u|+|\nabla u|_{\beta})\leq C$
and \eqref{7.2} implies that  $|\nabla u|_{\beta}\neq 0$. In particular, $u$ is not a constant.

Note that (see \cite{tw1305})
\begin{align}
\label{betaj}
\beta_j:=&C_{j}^2\Phi_{j}^{\ast}\alpha=\mn\alpha_{p\overline{q}}(x_j+C_j^{-1}z)\md z^p\wedge\md\overline{z}^q\longrightarrow \alpha (x)=\beta,\\
\label{7.6}&\Phi_{j}^{\ast}\alpha=C_{j}^{-2}\mn\alpha_{p\overline{q}}(x_j+C_j^{-1}z)\md z^p\wedge\md\overline{z}^q\longrightarrow  0\\
\label{7.7}
&\Phi_{j}^{\ast}\varpi=C_{j}^{-2}\mn\varpi_{p\overline{q}}(x_j+C_{j}^{-1}z,t)\md z^p\wedge\md \overline{z}^q\longrightarrow 0,
\end{align}
smoothly on compact sets of $\mathbb{C}^n$.
Also it is easy to get
\begin{align}\label{phizu}
\Phi_{j}^{\ast}Z(u_j)=&\mn\left( Z_{p\overline{q}}^ru_{j,r}
+\overline{Z_{q\overline{p}}^r}u_{j,\overline{r}}\right)(x_j+C_{j}^{-1}z)C_{j}^{-2}\md z^p\wedge\md \overline{z}^q\\
=&\mn\left( Z_{p\overline{q}}^r(x_j+C_{j}^{-1}z)\hat{u}_{j,r}(z)
+\overline{Z_{q\overline{p}}^r}(x_j+C_{j}^{-1}z)\hat{u}_{j,\overline{r}}(z)\right) C_{j}^{-1}\md z^p\wedge\md \overline{z}^q\longrightarrow  0\nonumber
\end{align}
uniformly   on compact sets of $\mathbb{C}^n$.
  Then we can deduce
\begin{align}
\label{7.4}\Phi_{j}^{\ast}\left((\Delta_{\alpha}u_j)\alpha\right)
=&\mn(\Delta_{\beta_j}\hat u_j)\beta_j\longrightarrow (\Delta u)\beta\quad \text{(see \cite{stw1503})}\\
\label{7.5}\Phi_{j}^{\ast}\ddbar u_j
=&\mn\frac{\partial^2 u_j}{\partial w^k\partial\overline{w}^{\ell}}(x_j+C^{-1}z)C_{j}^{-2}\md z^p\wedge\md \overline{z}^q\\
=&\mn\frac{\partial^2 \hat{u}_j}{\partial z^p\partial\overline{z}^{q}}(z) \md z^p\wedge\md \overline{z}^q\nonumber\\
=&\ddbar\hat{u}_j\longrightarrow \ddbar u\nonumber
\end{align}
weakly in $L_{\mathrm{loc}}^{p}(\mathbb{C}^n)$ of the coefficients. In particular, we have
\begin{align*}
\Phi_{j}^{\ast}\tilde\omega_j:=&\Phi_{j}^{\ast}\tilde\omega(t_j)=\Phi_{j}^{\ast}\left(\varpi(t_j)+\frac{1}{n-1}\left((\Delta_{\alpha}u_j)\alpha-
\ddbar u_j\right)+Z(u_j)\right)\\
&\longrightarrow\frac{1}{n-1}\left((\Delta u)\beta-\ddbar u\right)= P(u)
\end{align*}
weakly as currents. Since $\Phi_{j}^{\ast}\tilde\omega_j>0$, it follows that $P(u)\geq0$ as currents. From Lemma \ref{lempreliminary}, it follows that functions $\dot{u}_j\circ\Phi_j$ are uniformly bounded. Hence, we have
\begin{align}\label{tildeomegan}
\Phi^{\ast}\tilde\omega_{j}^n=&e^{\dot{u}_j\circ\Phi_j}\Phi_{j}^{\ast}\Omega\\
=&\frac{(\mn)^n}{C_{j}^{2n}}e^{\dot{u}_j\circ\Phi_j+\psi\circ\Phi_j}(\det\alpha)(x_j+C_{j}^{-1}z)\md z^1\wedge\md\overline{z}^1\wedge\cdots\wedge\md z^n\wedge\md \overline{z}^n\longrightarrow 0\nonumber
\end{align}
uniformly on compact sets of $\mathbb{C}^n$. From \eqref{betaj} and \eqref{7.4}, we have,
for any compact set $K\subset\mathbb{C}^n$,
$$
\sup_K|\Delta_{\beta_j}\hat u_j-\Delta\hat u_j|\longrightarrow 0,\quad\text{as}\quad j\longrightarrow \infty.
$$
From \eqref{betaj},  \eqref{7.6},  \eqref{phizu}, \eqref{7.4} and \eqref{7.5}, we have
\begin{align}\label{pujphitildeomega}
&\sup_K |P(\hat{u}_j)-\Phi_j^{\ast}\tilde{\omega}_j|_{\beta}\\
=& \sup_K\left|\frac{1}{n-1}\left((\Delta_\beta \hat{u}_j)\beta-(\Delta_{\beta_j}\hat{u}_j)\beta_j\right) -\Phi_j^{\ast}\varpi-\Phi_j^{\ast} Z(u_j)\right|_{\beta}\nonumber\\
\leq& \sup_K\left( \frac{1}{n-1}|(\Delta_\beta \hat{u}_j)\beta-(\Delta_{\beta_j}\hat{u}_j)\beta_j|_\beta
+|\Phi_j^{\ast}\varpi|_{\beta}+|\Phi_j^{\ast} Z(u_j)|_{\beta}\right)\nonumber
\end{align}
converges to zero as $j\longrightarrow \infty$. Noting that $P(\hat{u}_j)$ and $\Phi_j^{\ast}\tilde{\omega}_j$ are locally uniformly bounded, \eqref{tildeomegan} and \eqref{pujphitildeomega} implies that $P(\hat u_j)^n$  converges to zero uniformly on compact sets of $\mathbb{C}^n$, since we have $$
P(\hat u_j)^n-\Phi_j^{\ast}\tilde{\omega}_j^n=P(\hat u_j)^n-\left(\Phi_j^{\ast}\tilde{\omega}_j\right)^n
=\left(P(\hat u_j)-\Phi_j^{\ast}\tilde{\omega}_j\right)\sum\limits_{r=0}^{n-1}\left(P(\hat u_j)^r\wedge(\Phi_j^{\ast}\tilde{\omega}_j)^{n-1-r}\right).
$$
Then use the arguments in \cite{tw1305}, we know that $u$ is maximal. From Theorem \ref{ltthm}, we know that $u$ is a constant, a contradiction to $\nabla u (0)\neq 0$.
\end{proof}
\section{Proof of the uniqueness and long time existence of the main theorem}
\label{proofofuniqueandlong}
In this section, we will give the proof of the uniqueness and long time existence of Theorem \ref{mainthm}. To get this, we need the following lemma.
\begin{lem}
\label{lemhighorder}
Let   $u$ be the solution to \eqref{paragau} on $M\times[0,\,T)$. Then for any $\epsilon \in (0,\,T)$ and every positive integer $k$, there is a constant $C_k$, depending only on $k,\,\epsilon$ and the initial data on $M$ such that
\begin{align}
\sup\limits_{M\times[\epsilon,\,T)}|\nabla^ku(x,\,t)|\leq C_k.
\end{align}
\end{lem}
\begin{proof}
Thanks to Theorem \ref{2ndestimate} and Theorem \ref{1stestimate}, we can deduce a priori $C^{2,\gamma}$ estimates of $u$ from \cite[Theorem 5.3]{chujianchuncvpde} and the discussion in the proof of \cite[Lemma 6.1]{chu1607} directly (cf. \cite{TWWY}).  Higher order estimates can be obtained after differentiating the equation and applying the usual bootstrapping method.
\end{proof}
Then we can prove the long time existence and uniqueness of the solution in the main theorem.
\begin{proof}[Proof of the uniqueness and long time existence of Theorem \ref{mainthm}]
The uniqueness is the result of standard parabolic equation theory. Suppose that $T<\infty$. From Lemma \ref{lempreliminary}, there exists a uniform constant $C$ such that
\begin{align}
|u(x,t)|\leq |u_0(x)|+T\sup\limits_{M\times [0,\,T)}|\dot u|\leq C(T+1)<\infty.
\end{align}
This together with Lemma \ref{lemhighorder} and short time existence theorem implies that we can extend the solution to \eqref{paragau} on $[0,T+\varepsilon)$ with $\varepsilon>0$, absurd. We can find more details about this standard discussion in the proof of \cite[Theorem 3.1]{tosattikawa} (cf.\cite{songweinkovekrf,weinkovekrf}).
\end{proof}
\section{The Harnack inequality}
\label{harnack}
In this section we consider the Harnack inequalities of the parabolic equation
\begin{equation}\label{parahar}
\ppt \varphi=L(\varphi)
\end{equation}
where $L$ is defined as in \eqref{Ldefn}, which are analogs to Li and Yau \cite{LY}. Note that Cao \cite{Cao} stated this result on K\"{a}hler manifolds and Gill \cite{gill} also proved it on Hermitian manifolds (cf. \cite{chu1607,W}). This is another necessary preparation for the proof of our main theorem. For convenience, we give a lemma which can be easily obtained from Lemma \ref{lemhighorder} as follows.
\begin{lem}
\label{lemhighorder2}
Let $u$ be the solution to \eqref{paragau} on $M\times[0,\,\infty)$. Then for every positive integer $k$, there is a constant $C_k$, depending only on $k$ and the initial data on $M$ such that
\begin{equation*}
\sup\limits_{M\times[0,\,\infty)}|\nabla^ku(x,\,t)|\leq C_k.
\end{equation*}
\end{lem}
Then we give our Harnack inequality as follows.
\begin{lem}
\label{harnacklem1}
Assume that $\varphi$ is a positive solution to \eqref{parahar} and define $f=\log \varphi$
and $F=t\left(|\partial f|^2-\alpha f_t\right)$. There holds
\begin{align}
\label{LF-Ft}
L(F)-F_t\geq \frac{1}{2n}\left(|\partial f|^2-f_t\right)^2
-2\mathrm{Re}\left\langle \partial f, \partial F\right\rangle
-\left(|\partial f|^2-\alpha f_t\right)
-Ct|\partial f|^2-Ct,
\end{align}
where the definition of $|\partial f|^2$ can be found in the proof.
\end{lem}
\begin{proof}
The ideas are originated from Li and Yau \cite{LY}, hence we can be brief and just point out the main differences.
For convenience, we introduce some notations as follows.
\begin{align*}
\langle \chi,\,\xi\rangle:=\Theta^{\overline{j}i}\chi_i\overline{\xi_j},\quad
|\partial f|^2:=\langle \partial f,\,\partial f\rangle, \quad
|\partial \overline{\partial} f|^2:=\Theta^{\overline{j}i}\Theta^{\overline{\ell}k}f_{i\overline{\ell}}f_{\overline{j}k},\quad
|D^2f|^2:=\Theta^{\overline{j}i}\Theta^{\overline{\ell}k}f_{ik}f_{\overline{j}\overline{\ell}},
\end{align*}
where $\chi$ and $\xi$ are $(1,0)$ forms.
By direct computation, we get
$$
L(f)-f_t=-|\partial f|^2,
$$
i.e.,
\begin{equation}\label{FandL}
F=-tL(f)-t(\alpha-1)f_t.
\end{equation}
Then we can deduce
\begin{equation*}
(L(f))_t=\frac{1}{t^2}F-\frac{1}{t}F_t-(\alpha-1)f_{tt}.
\end{equation*}
Again direct computation implies
\begin{equation}\label{Ft}
F_t=|\partial f|^2-\alpha f_t
+2t\mathrm{Re}\left\langle\partial f,\,\partial f_t\right\rangle
+t\left(\ppt\Theta^{\overline{j}i}\right)f_if_{\overline{j}}
-\alpha t f_{tt},
\end{equation}
and
\begin{align}
\label{Fklbar}
F_{k\overline{\ell}}
=&t\left[\left(\Theta^{\overline{j}i}\right)_{k\overline{\ell}}f_if_{\overline{j}}
+\left(\Theta^{\overline{j}i}\right)_{k}f_{i\overline{\ell}}f_{\overline{j}}
+\left(\Theta^{\overline{j}i}\right)_{k}f_if_{\overline{j}\overline{\ell}}\right.\\
&+\left(\Theta^{\overline{j}i}\right)_{\overline{\ell}}f_{ik}f_{\overline{j}}
+\Theta^{\overline{j}i}f_{ik\overline{\ell}}f_{\overline{j}}
+\Theta^{\overline{j}i}f_{ik}f_{\overline{j}\overline{\ell}}\nonumber\\
&+\left.\left(\Theta^{\overline{j}i}\right)_{\overline{\ell}}f_if_{\overline{j}k}
+\Theta^{\overline{j}i}f_{i\overline{\ell}}f_{\overline{j}k}
+\Theta^{\overline{j}i}f_if_{\overline{j}k\overline{\ell}}-\alpha f_{tk\overline{\ell}}
\right].\nonumber
\end{align}
Note that
\begin{align}
\label{3orderdao}
&t\Theta^{\overline{\ell}k} \Theta^{\overline{j}i}f_{ik\overline{\ell}}f_{\overline{j}}
+ t\Theta^{\overline{\ell}k} \Theta^{\overline{j}i}f_if_{\overline{j}k\overline{\ell}}\\
=&2t\mathrm{Re}\left\langle \partial f, \partial\left(\Theta^{\overline{\ell}k}f_{k\overline{\ell}}\right)\right\rangle
-t\Theta^{\overline{j}i}\left(\Theta^{\overline{\ell}k}\right)_if_{k\overline{\ell}}f_{\overline{j}}
-t\Theta^{\overline{j}i}\left(\Theta^{\overline{\ell}k}\right)_{\overline{j}}f_{k\overline{\ell}}f_{i}\nonumber\\
\geq&2t\mathrm{Re}\left\langle \partial f, \partial\left(\Theta^{\overline{\ell}k}f_{k\overline{\ell}}\right)\right\rangle
-\frac{C_1t}{\varepsilon}|\partial f|^2-\varepsilon t |\partial\overline{\partial}f|^2\nonumber\\
=&-2\mathrm{Re}\left\langle \partial f, \partial F\right\rangle
-2t\mathrm{Re}\left\langle \partial f, \partial \left(\tr_{\tilde \omega}Z(f)\right)\right\rangle
-2t(\alpha-1)\mathrm{Re}\left\langle\partial f,\,\partial f_t\right\rangle
-\frac{C_1t}{\varepsilon}|\partial f|^2-\varepsilon t |\partial\overline{\partial}f|^2\nonumber\\
=&-2\mathrm{Re}\left\langle \partial f, \partial F\right\rangle
-2t\mathrm{Re}\left\langle \partial f, \partial \left(\tr_{\tilde \omega}Z(f)\right)\right\rangle
-\frac{C_1t}{\varepsilon}|\partial f|^2-\varepsilon t |\partial\overline{\partial}f|^2\nonumber\\
&-(\alpha-1)F_t
+(\alpha-1)\left(|\partial f|^2-\alpha f_t\right)
+(\alpha-1)t\left(\ppt\Theta^{\overline{j}i}\right)f_if_{\overline{j}}
-(\alpha-1)\alpha tf_{tt}\nonumber\\
\geq&-2\mathrm{Re}\left\langle \partial f, \partial F\right\rangle
-\frac{C_2t}{\varepsilon}|\partial f|^2
-\varepsilon t |\partial\overline{\partial}f|^2
-\varepsilon t |D^2f|^2-C_3t|\partial f|^2\nonumber\\
&-(\alpha-1)F_t+(\alpha-1)\left(|\partial f|^2-\alpha f_t\right)
-\alpha(\alpha-1) tf_{tt},\nonumber
\end{align}
where for the first inequality we use Young's inequality, for the second equality we use \eqref{FandL}, for the third equality we use \eqref{Ft}, and for the second inequality we use Lemma \ref{lemhighorder2}, Young's inequality and the facts that $Z(f)$ is linear of $\partial f$ and $\overline{\partial}f$.

Also using Young's inequality, we have
\begin{align}
\label{ftt}
-\alpha t\Theta^{\overline{\ell}k}f_{tk\overline{\ell}}
=&\alpha t\left(\ppt\Theta^{\overline{\ell}k}\right)f_{k\overline{\ell}}
-\alpha t\ppt\left(\Theta^{\overline{\ell}k}f_{k\overline{\ell}}\right)\\
\geq &-\frac{C_4t}{\varepsilon}-\varepsilon t|\partial\overline{\partial}f|^2
+\alpha t\ppt \left(\tr_{\tilde\omega}Z(f)\right)
-\frac{\alpha}{t}F+\alpha F_t+\alpha(\alpha-1)tf_{tt}\nonumber\\
\geq&-\frac{C_5t}{\varepsilon}
-\varepsilon t|\partial\overline{\partial}f|^2
-C_6t|\partial f|^2
+\alpha t\left(\tr_{\tilde\omega}Z(f_t)\right)
-\frac{\alpha}{t}F+\alpha F_t+\alpha(\alpha-1)tf_{tt},\nonumber
\end{align}
where for the second inequality we also use Lemma \ref{lemhighorder2}.

According to the definition of $F$, Lemma \ref{lemhighorder2} and Young's inequality,  we can deduce
\begin{align}
\label{1oderterm}
\tr_{\tilde\omega}Z(F)=&t \left(\tr_{\tilde\omega}Z\left(\Theta^{\overline{j}i}f_if_{\overline{j}}\right)\right)
-\alpha t\left(\tr_{\tilde\omega} Z(f_t)\right)\\
\geq&-\frac{C_7t}{\varepsilon}|\partial f|^2
-C_8t|\partial f|^2
- \varepsilon t |\partial\overline{\partial }f|^2
-\varepsilon t|D^2 f|^2
-\alpha t\left(\tr_{\tilde\omega} Z(f_t)\right).\nonumber
\end{align}
Again using Young's inequality, we get
\begin{align}
\label{qitaterm}
&t\Theta^{\overline{\ell}k}
\left[\left(\Theta^{\overline{j}i}\right)_{k\overline{\ell}}f_if_{\overline{j}}
+\left(\Theta^{\overline{j}i}\right)_{k}f_{i\overline{\ell}}f_{\overline{j}}
+\left(\Theta^{\overline{j}i}\right)_{k}f_if_{\overline{j}\overline{\ell}}
+\left(\Theta^{\overline{j}i}\right)_{\overline{\ell}}f_{ik}f_{\overline{j}}
+\left(\Theta^{\overline{j}i}\right)_{\overline{\ell}}f_if_{\overline{j}k}\right]\\
\geq&-C_9t|\partial f|^2-\frac{C_{10}t}{\varepsilon}-\varepsilon t|\partial\overline{\partial}f|^2-\varepsilon t|D^2f|^2.\nonumber
\end{align}
Combining \eqref{Ft}, \eqref{Fklbar}, \eqref{3orderdao}, \eqref{ftt}, \eqref{1oderterm} and \eqref{qitaterm}, we can deduce
\begin{align}
\label{LF-Ft1}
L(F)-F_t
\geq&-\frac{C_5+C_{10}}{\varepsilon}t-2\mathrm{Re}\left\langle \partial f, \partial F\right\rangle
+(1-4\varepsilon)t|\partial\overline{\partial}f|^2
+(1-3\varepsilon)t|D^2f|^2\\
&+\left(C_3+C_6+C_8+C_9+\frac{C_2+C_7}{\varepsilon}\right)t|\partial f|^2
-\left(|\partial f|^2-\alpha f_t\right).\nonumber
\end{align}
From the Cauchy-Schwarz inequality, we have
\begin{align}
\label{cr}
|\partial\overline{\partial}f|^2
\geq&\frac{1}{n}\left(\Theta^{\overline{j}i}f_{i\overline{j}}\right)^2
=\frac{1}{n}\left(|\partial f|^2-f_t+\tr_{\tilde\omega}Z(f)\right)^2\\
\geq&\frac{1}{n}(1-\varepsilon)\left(|\partial f|^2-f_t\right)^2-C\varepsilon|\partial f|^2,\nonumber
\end{align}
where for the second inequality we use Young's inequality.

Then from \eqref{LF-Ft1} and \eqref{cr}, we can deduce \eqref{LF-Ft}.
\end{proof}
Using Lemma \ref{harnacklem1} and standard discussion, we can get
\begin{lem}
\label{harnacklem2}
Using the notations as in Lemma \ref{harnacklem1}, for any $t\in[0,\,\infty)$, there exist uniform constants $C_1$ and $C_2$ such that
\begin{align}
|\partial f|^2-\alpha f_t\leq C_1+\frac{C_2}{t}.
\end{align}
\end{lem}
Then by this lemma we can deduce the following theorem which will be used in the convergence discussion.
\begin{thm}
\label{harnackthm}
Using the notations as in Lemma \ref{harnacklem1}, for any $0<t_1<t_2$, there holds
\begin{align}
\sup_{x\in M}\varphi(x,t_1)\leq \inf_{x\in M}\varphi(x,t_2)\left(\frac{t_2}{t_1}\right)^{C_1}
\exp\left(\frac{C_2}{t_2-t_1}+C_3(t_2-t_1)\right),
\end{align}
where $C_1,\;C_2$ and $C_3$ are uniform constants.
\end{thm}
We remark that the proofs of Lemma \ref{harnacklem2} and Theorem \ref{harnackthm} are standard discussion. Here considering the length of the paper, we omit them (see for example \cite{chu1607,gill}).
\section{Proof of the convergence in the main theorem}
\label{convergence}
In this section, we will give the proof of the second part of Theorem \ref{mainthm}. Since the discussion is standard, we will be brief and just point out the main differences.
\begin{proof}[Proof of the convergence of Theorem \ref{mainthm}]
For any positive integer $m$, we define
\begin{align*}
\xi_{m}(x,t)=&\sup_{y\in M}\dot u(y,m-1)-\dot u(x,m-1+t)\\
\psi_{m}(x,t)=&\dot u(x,m-1+t)-\inf_{y\in M}\dot u(y,m-1).
\end{align*}
These functions satisfy the equation \eqref{parahar}.

Without loss of generality, we assume that $\dot u(x,m-1)$ is not constant. The maximum principle implies that
$\xi_m$ and $\psi_m$ are positive solutions to \eqref{parahar}. Using Theorem \ref{harnackthm} with $t_1=1/2$ and $t_2=1$, we get
\begin{align*}
\sup_{y\in M}\dot u(y,m-1)-\inf_{y\in M}\dot u(x,m-1/2)
\leq C\left(\sup_{y\in M}\dot u(y,m-1)-\sup_{y\in M}\dot u(x,m)\right)\\
\sup_{y\in M}\dot u(y,m-1/2)-\inf_{y\in M}\dot u(x,m-1)
\leq C\left(\inf_{y\in M}\dot u(y,m)-\inf_{y\in M}\dot u(x,m-1)\right),
\end{align*}
which implies
\begin{align*}
\theta(m-1)+\theta(m-1/2)\leq C\left(\theta(m-1)-\theta(m)\right),
\end{align*}
where $\theta(t)=\sup_{y\in M}\dot u(y,t)-\inf_{y\in M}\dot u(y,t)$.
Then by induction, we get
\begin{align}
\theta(t)\leq Ce^{-\eta t},
\end{align}
where $\eta=\log\frac{C}{C-1}$.

Since $\int_M\tilde u\alpha^n=0$, we obtain $\int_M\frac{\partial \tilde u}{\partial t}\alpha^n=0$, which implies that
there exists a point $y\in M$ with $\frac{\partial\tilde u}{\partial t}(y,t)=0$. Then we have
\begin{align}
\label{tildeut}
\left|\frac{\partial \tilde { u}}{\partial t}(x,t)\right|
=&\left|\frac{\partial \tilde { u}}{\partial t}(x,t)-\frac{\partial \tilde { u}}{\partial t}(y,t)\right|\\
=&\left|\frac{\partial u}{\partial t}(x,t)-\frac{\partial u}{\partial t}(y,t)\right|\nonumber\\
\leq&Ce^{-\eta t}.\nonumber
\end{align}
Consider the quantity $Q=\tilde u+\frac{C}{\eta}e^{-\eta t}$. Then we have
$$
\frac{\partial Q}{\partial t}=\frac{\partial \tilde u}{\partial t}-Ce^{-\eta t}\leq 0,
$$
which implies that $ Q$ converges uniformly as $t\rightarrow \infty$ since $Q$ is uniformly bounded and denote the limit by $\tilde u_{\infty}$.
Furthermore, we have
$$
\lim\limits_{t\rightarrow\infty}\tilde u
=\lim\limits_{t\rightarrow\infty}Q-\lim\limits_{t\rightarrow\infty}\frac{C}{\eta}e^{-\eta t}=\tilde u_{\infty}.
$$
It follows that the convergence of $\tilde u$ to $\tilde u_{\infty}$ is actually $C^{\infty}$ using the argument by contradiction and the Arzela-Ascoli theorem.
Note that
$$
\frac{\partial \tilde u}{\partial t}=
\log\frac{\left(\varpi+\frac{1}{n-1}\left[(\Delta \tilde u)\alpha-\ddbar \tilde u\right]+Z(\tilde u)\right)^n}{\alpha^n}-\psi-\frac{1}{\mathrm{Vol}_{\alpha}(M)}\int_M\frac{\partial u}{\partial t}\alpha^n.
$$
Letting $t\longrightarrow \infty$, \eqref{tildeut} implies
$$
\log\frac{\left(\varpi+\frac{1}{n-1}\left[(\Delta \tilde u_{\infty})\alpha-\ddbar \tilde u_{\infty}\right]+Z(\tilde u_{\infty})\right)^n}{\alpha^n}=\psi+\tilde b,
$$
where
\begin{equation}
\label{btilde}
\tilde b=\frac{1}{\mathrm{Vol}_{\alpha}(M)}\int_M\left(\log\frac{\left(\varpi+\frac{1}{n-1}\left[(\Delta \tilde u_{\infty})\alpha-\ddbar \tilde u_{\infty}\right]+Z(\tilde u_{\infty})\right)^n}{\alpha^n}-\psi\right)\alpha^n,
\end{equation}
as required.

\end{proof}

\end{CJK}
\end{document}